\documentclass[reqno,onecolumn,oneside]{paper}
\usepackage[colorlinks]{hyperref}
\usepackage{color}
\usepackage[english]{babel}
\usepackage{enumerate,
amsmath,
amsthm,
amsfonts,
pifont,
slashed,
ifsym,
yfonts,
calligra,
amssymb,
latexsym,
mathrsfs,
}
\usepackage[all]{xy}
\usepackage[sort]{cite}

\newtheorem{theorem}{Theorem}[section]
\newtheorem{proposition}[theorem]{Proposition}
\newtheorem{lemma}[theorem]{Lemma}
\newtheorem{corollary}[theorem]{Corollary}
\theoremstyle{definition}
\newtheorem{definition}[theorem]{Definition}

\theoremstyle{remark}

\newtheorem*{remark*}{Remark}

\makeatletter
\newdimen\mymathindent
\newenvironment{bulletequation}%
    {\@beginparpenalty\predisplaypenalty
     \@endparpenalty\postdisplaypenalty
     \refstepcounter{equation}%
     \trivlist \item[]\leavevmode
       \hb@xt@\linewidth\bgroup $\m@th
         \displaystyle
         \hskip\mymathindent}%
        {$\hfil 
         \displaywidth\linewidth\hbox{\@eqnnum}%
       \egroup
     \endtrivlist}
\makeatother

\newcommand{\cns}{\cat C\!\mathit{ns}}
\newcommand{\bydef}{\mathrel{\mathop:}=}

\newcommand{\id}{\operatorname{id}}

\newcommand{\sub}{(~~)}
\newcommand{\subl}{(~~)^l}
\newcommand{\subr}{(~~)^r}
\newcommand{\<}{_\textrm{--}}
\newcommand{\ID}{\operatorname{\textsc{Id}}}
\newcommand{\T}{\operatorname{Thm}}

\newcommand{\lang}{\mathcal{L}}
\newcommand{\fml}{\mathit{Fm}}
\newcommand{\Fml}{\mathbf{Fm}}
\newcommand{\fmll}{\mathit{Fm}'}
\newcommand{\Fmll}{\mathbf{Fm'}}

\newcommand{\sfm}{\Sigma_{\cat L}}
\newcommand{\Sfm}{\mathbf{\Sigma_{\cat L}}}
\newcommand{\sfmm}{\Sigma_{\cat L'}}
\newcommand{\Sfmm}{\mathbf{\Sigma_{\cat L'}}}

\newcommand{\SL}{\mathcal{S}\ell}
\newcommand{\Q}{\mathcal{Q}}

\newcommand{\Mod}{\textrm{-}\mathcal{M}\!\!\:\mathit{od}}

\newcommand{\cat}{\mathcal}

\newcommand{\MM}{\mathbf{M}}
\renewcommand{\AA}{\mathbf{A}}
\newcommand{\NN}{\mathbf{N}}
\newcommand{\LL}{\mathbf{L}}

\newcommand{\QQ}{\mathbf{Q}}
\newcommand{\RR}{\mathbf{R}}
\newcommand{\TT}{\mathbf{T}}
\newcommand{\VV}{\mathbf{V}}
\newcommand{\BB}{\mathbf{B}}

\newcommand{\V}{\mathit{Var}}
\newcommand{\EQ}{\mathbf{Eq}}
\newcommand{\Eq}{\mathit{Eq}}

\newcommand{\seq}{\mathit{Seq}}
\renewcommand{\SS}{\mathbf{S}}
\newcommand{\Hom}{\mathbf{hom}}

\newcommand{\N}{\mathbb{N}}

\renewcommand{\wp}{\mathscr{P}}

\newcommand{\Th}{\mathbf{Th}}

\renewcommand{\phi}{\varphi}

\newcommand{\g}{\gamma}
\renewcommand{\d}{\delta}

\newcommand{\restr}{\upharpoonright}
\newcommand{\la}{\left\langle}
\newcommand{\ra}{\right\rangle}

\newcommand{\under}{\backslash}
\newcommand{\ost}{{}_\cdot/}
\newcommand{\lto}{\longrightarrow}
\newcommand{\To}{\Rightarrow}

\newcommand{\lmapsto}{\longmapsto}
\newcommand{\ust}{\under_\cdot}
\newcommand{\ov}{\overline}

\newcommand{\tensor}{\otimes}

\newcommand{\eq}{{\approx}}

\newcommand{\f}{\ensuremath{\varphi}}

\def\amslatex\slash{{\protect\AmS-\protect\LaTeX}}

\begin{document} 
\title{An order-theoretic analysis of interpretations among propositional deductive systems}

\author{Ciro Russo}
\institution{Departamento de Matem\'atica \\ Universidade Federal da Bahia -- BA, Brazil \\ \small{\texttt{ciro.russo@ufba.br}}}
\maketitle
\date{Sep 21, 2012}

\begin{abstract}
In this paper we study interpretations and equivalences of propositional deductive systems by using a quantale-theoretic approach introduced by Galatos and Tsinakis. Our aim is to provide a general order-theoretic framework which is able to describe and characterize both strong and weak forms of interpretations among propositional deductive systems also in the cases where the systems have different underlying languages.

\vskip 10pt

\textbf{Ackowledgements.} I am extremely grateful to Constantine Tsinakis for his participation in the development of this work. In my opinion, his role in the preparation of this paper deserved a coauthorship, which he unfortunately declined.
\end{abstract}

\section*{Introduction}

The problem of comparing logical systems can be traced back to the early twentieth century, when Brouwer introduced intuitionistic logic \cite{brou}. In the debate around the principle of excluded middle, fostered by Brouwer's ideas, a central question was whether the provable assertions of classical logic could also be formulated and proved in intuitionistic logic. Such a question led to the problem of interpreting classical logic into intuitionistic logic.

Starting from the assumption that ``it is illegitimate to use the principle of excluded middle in the domain of transfinite arguments'', Kolmogorov \cite{kolmo} proved~---~in 1925~---~that ``finitary conclusions obtained by means of the principle of excluded middle, are in fact correct and can be proved even without its use.'' The main result he obtained essentially asserts that any provable formula of classical logic is intuitionistically provable provided that all of its subformulas are replaced by their respective double negations. From the viewpoint of propositional calculus, the Kolmogorov interpretation is not invariant with respect to the action of substitutions: if the interpretation and a substitution are applied to a classical formula, the resulting intuitionistic formula depends on which of the two is applied first. In 1929 Glivenko \cite{glivenko} proved the following theorem: ``An arbitrary propositional formula $A$ is classically provable, if and only if $\lnot\lnot A$ is intuitionistically provable.'' A major difference between the two interpretations is their behavior with respect to substitutions, as the latter is invariant with respect to any substitution in the language of classical logic.

In 1934 Gentzen introduced the sequent calculi LK and LJ, for classical and intuitionistic logic respectively \cite{gent}. Such new formal systems, which gave birth to proof theory and automated deduction, offered a new perspective on the relationship between the two logics. Indeed the two systems differ from each other just in the type of sequents they can handle~---~LJ admitting exclusively sequents with a single formula or no formulas on the right-hand side (and the obvious adjustments of the rules of LK reflecting the change of the sequent-type).

The existence of algebraic semantics for certain logics can be ascribed to Lindenbaum and Tarski, who showed how it is possible to associate in a canonical way, at least at the propositional level, logical calculi (and their attendant consequence relations) with classes of algebras.  We refer the reader to the article \cite{mpt} for a detailed survey of these developments. Moreover, in \cite{tarski}, Tarski proposed two methods for attacking the \emph{decision problem} in first-order logical systems, one of which~--- the so-called ``indirect method''~--- consists in transferring the problem from a system to some other one for which it has previously been solved; in order to do that, Tarski defined the concept of translation between (first-order) logical systems.

Many years later, Blok and Pigozzi introduced the concepts of \emph{equivalent algebraic semantics} and \emph{algebraizable logic} \cite{blokpigozzi}, which requires the comparison of two consequence relations defined on different syntactic constructs (formulas, equations, sequents) and, in some cases, also on different languages. As a result of that work, the interest for interpretations and translations between logics increased rapidly in the last two decades, and many authors investigated this problem from various points of view and with different approaches; see, for instance, \cite{blokjonsson2, blokpigozzi2, cze, dott, dza, feitosa, galtsi, pynko, raftery, rebver, woj}. These studies have produced a wide range of concepts connected to interpretations and translations of logics. 

In particular, it is shown in \cite{galtsi} that consequence relations (defined on sets of formulas, equations or sequents) can be represented as \emph{structural closure operators} on quantale modules --- such operators being in bijective correspondence with quotiens for any given module.

The aim of this work is to study interpretations and translations between propositional logics using the algebraic techniques developed in \cite{galtsi, thesis,russo}. A distinct feature of our work is the separation of the concept of a \emph{translation} from that of an \emph{interpretation}. Indeed, to our knowledge, since the aforementioned works by Kolmogorov, Glivenko and Tarski appeared, a translation of a language into another has only been considered as a part of an interpretation of one logical system into another~---~in some cases the words ``interpretation'' and ``translation'' being actually used as synonyms. On the contrary, here we will consider translations between languages as objects of study themselves, regardless~--- a priori~--- of whether an interpretation exists. A similar approach can be found in \cite{diac}, where language translations are called \emph{signature morphisms}.

Indeed, the fact that a deductive system is interpretable in another one means, roughly speaking, that the two consequence relations ``agree'' at some level with each other. On the other hand, the language is a syntactic object whose existence is independent from any consequence relation that, eventually, can be defined on it; indeed, we may have many different consequence relations defined on the same language. For these reasons, it is preferable that a translation between two languages regards only the connectives of the languages without any a priori involvement of the deductive apparatus of the systems. To illustrate this point, if two people speak different languages and one of them is able to translate in his own language what the other says, it does not follow that he/she also agrees with his/her interlocutor's ideas. Conversely, two people may have identical ideas but be unable to translate them into each other's language.

According to this point of view, one can encounter different situations corresponding to the existence and non existence of a translation of languages and the interpretation of two consequence relations, and to the strengthening and weakening of the concept of interpretation.

On the other hand, logical systems with different underlying languages are, in some sense (that will appear clearer and more precise later on in the paper), objects in different categories. Therefore, a comparison between such systems requires the existence of a canonical method for putting them in the same category, namely, the existence of a suitable functor.

The main results of the paper can be briefly summarized as follows.
\begin{itemize}
\item In Lemma \ref{qinq'} we show that any translation between two given propositional languages $\lang$ and $\lang'$ induces a homomorphism between the corresponding substitution monoids. This result is completed by Theorem \ref{transchar}, in which we characterize the homomorphisms between substitution monoids that are induced by translations.
\item Theorem \ref{retraction} shows that surjectivity is a sufficient condition for a language translation to have a right-inverse translation and, therefore, to induce a monoid retraction.
\item In Theorem \ref{gt} we extend one of the main results of \cite{galtsi} to interpretations and representations between deductive systems with the same language.
\item In Section \ref{tens} we show that several ring-theoretic constructions can be suitably adapted to quantale modules. In particular, in Theorem \ref{adjfunct}, we prove that any homomorphism between two quantales defines an adjoint and co-adjoint functor (in the opposite direction) between the corresponding categories of modules. Moreover, we prove that such a functor is also a full embedding if the corresponding quantale homomorphism is surjective, and its left adjoint is a retraction of categories if the quantale homomorphism is a retraction (Theorem \ref{subs}).
\item The constructions and results of Section \ref{tens}, together with the basic properties of quantales and with the results of Section \ref{transsec}, provide the desired functors that allow us to sensibly generalize Theorem \ref{gt} and several results of \cite{galtsi}. In particular, Theorem \ref{faiththm} and Corollary \ref{equivthm} provide the extension of Theorem \ref{gt} to the case of systems with different languages; Theorem \ref{equivthm2} proves that, assuming the existence of a surjective translation, the characterization of Corollary \ref{equivthm} can be eased by using Theorem \ref{retraction}; last, Theorem \ref{nonconsthm} gives an account of weak interpretations.
\end{itemize}

\section{Abstract consequence relations}
\label{prel}

A successful approach to logical consequence, dating back at least to the work of Tarski,\footnote{The approach, in fact, has an antecedent in Bernard Bolzano's refined analysis of consequence in his \emph{Wissenschaftslehre} (1837).} consists in giving an account of it via a \emph{relation}. While Tarski defined consequence relation on an algebra of formulas, at the present stage we prefer to follow the approach of Blok and J\'onsson \cite{blokjonsson2}, and consider the more general case of a nonempty set $S$ about which no inner structure is postulated.

\begin{definition}
A \emph{consequence relation} over the set $S$ is a relation $\vdash \subseteq \wp (
S) \times S$ satisfying, for all $X, Y, \{u\} \subseteq S$:
\begin{itemize}
\item[] 
\begin{bulletequation} 
\textrm{if $u\in X$, then $X\vdash u$;} \label{sentailin}
\end{bulletequation}
\item[]  
\begin{bulletequation} 
\textrm{if $X\vdash u$ and $X\subseteq Y$, then $Y\vdash u$; and} \label{sentailimp}
\end{bulletequation}
\item[]  
\begin{bulletequation} 
\textrm{if $Y\vdash u$ and $X\vdash v$ for every $v\in Y$, then $X\vdash u$.} \label{sentailjoin}
\end{bulletequation}
\end{itemize}
\end{definition}
\noindent Following convention, we write $X \vdash Y$ if $X \vdash v$ for all $v \in Y$, 
and $\vdash v$ if $\emptyset \vdash v$.

An equivalent approach, also due to Alfred Tarski, consists in describing logical consequence by means of a closure operator. In the present context, we use the term \emph{consequence operator} for a closure operator on the power set $\wp(S)$ of a set $S$, that is, a map $\cns : \wp(S) \lto \wp(S)$ satisfying the following conditions for all $X, Y \subseteq S$: 

\begin{itemize}
\item[] 
\begin{bulletequation} 
\textrm{if $X \subseteq Y,$ then $\cns(X) \subseteq \cns(Y); $} 
\end{bulletequation}
\item[]  
\begin{bulletequation} 
\textrm{$X \subseteq \cns(X);$ and}
\end{bulletequation}
\item[]  
\begin{bulletequation} 
\textrm{$\cns (\cns(X)) = \cns(X).$}
\end{bulletequation}
\end{itemize}

Given a set $S$, there exists a bijective correspondence between all consequence operators $\cns $ on $\wp(S)$ and all consequence relations $\vdash $ over $S$. More specifically:

\begin{lemma}\label{oprel}
Let $S\neq \emptyset $ be a set. If $\vdash $ is a consequence relation over $S$, then the map $\cns_{\vdash}: \wp(S) \lto \wp(S)$
defined by
\begin{equation*}
\cns_{\vdash}(X) =\left\{ u\in S \mid X\vdash u\right\} 
\end{equation*}
is a consequence operator on $\wp(S) $.
Conversely, if $\cns$ is a consequence operator on $\wp(S) $, then the relation $\vdash_{\cns} \subseteq \wp(S) \times S$ defined by
\begin{equation*}
X \vdash _{\cns} u \ \text{ iff } \ u \in \cns(X) 
\end{equation*}
is a consequence relation over $S$. Furthermore, $\cns_{\vdash_{\cns}} = \cns$ and $\vdash_{\cns_{\vdash}} = \ \vdash $.
\end{lemma}

Let $\vdash $ be a consequence relation over $S$, and let $\cns$ be the associated consequence operator on $\wp(S)$. $X\subseteq S$ is said to be a \emph{$\vdash $-theory} if it is closed under $\cns$: $X= \cns(X)=\{u \in S \mid  X\vdash u\}$. Note that the poset of $\vdash$-theories, denoted by \emph{$\mathrm{Th}\left( \vdash \right)$} or \emph{$\mathrm{Th}\left( \cns \right)$}, is a closure system over $S$, that is, a subset of $\wp(S)$ that is closed under arbitrary intersections. $\mathrm{Th}\left( \vdash \right)$ completely determines $\cns$ and $\vdash$. Furthermore,  there exists a bijective correspondence between consequence relations over $S$, closure systems over $S$, and consequence operators on $\wp(S)$.

It is pertinent to remark that we have placed no restrictions on the cardinalities of our sets of premisses of a consequence relation, which may be finite as well as infinite. Nonetheless, since logical deductions generally proceed from finitely many premisses, we single out finitary consequence relations. Formally, a consequence relation $\vdash$  over $S$ is called \emph{finitary}, provided for all $X \cup \{u\} \subseteq S$,
\begin{equation}\label{sentailfinitary}
\textrm{if $X \vdash u$, then there exists a finite $Y \subseteq X$ such that $Y \vdash u$.}
\end{equation}
Note that $\vdash$ is finitary iff the associated consequence operation $\cns$ satisfies a related condition for all $X\subseteq S$ and all $u\in S$: if $u\in \cns(X)$, then $u\in \cns(Y),$ for some finite subset $Y$ of $X$. We use the term \emph{finitary} (instead of the more common term \emph{algebraic}) for any consequence operation that satisfies the preceding condition. However, for the sake of readability, such a restriction shall not be imposed throughout the paper; we will briefly discuss it in Section \ref{concl}.

One of the most distinctive features of logical consequence is its formal character: a logical system is built upon a given language and its consequence relation is preserved under the application of suitable syntactic modifications, called substitutions. Actions of monoids on sets provide a suitable mathematical framework for capturing this feature.  

Formally, let $S$ be a nonempty set. A monoid $\mathbf{A} = \la A, \cdot, 1\ra $ is said to \emph{act} on $S$ (and $S$
is said to be an \emph{$\mathbf{A}$-set}) in case there exists an operation $\cdot: A\times S\lto S$ that satisfies
\begin{equation}\label{action}
\left(a b\right) \cdot u = a \cdot \left(b \cdot u\right) \ \text{ and } \ 1 \cdot u = u,
\end{equation}
for all $a, b\in A$ and all $u\in S$. The action is also called \emph{scalar product} or \emph{scalar multiplication}.

Even if we used the same symbol $\cdot$ for both the monoid multiplication and the action of $\AA$ on $S$, in expressions like the ones in (\ref{action}) we use plain juxtaposition in place of the former and ``$\cdot $'' for the latter. We will keep using this convention (also for quantale modules, from Section \ref{ressec} on) throughout the paper whenever no confusion is likely to arise; moreover, in the case of different sets subject to monoid actions, we may use suitable subscripts in order to avoid confusion.

Let now $S$ be an $\AA$-set. A consequence relation $\vdash$ over $S$ is said to be 
\emph{action-invariant} if, for any $a\in A$ and any $X\cup\{u\} \subseteq S$, 
\begin{equation}\label{sentailstructural}
\text{whenever $X \vdash u$, then $a \cdot X \vdash a \cdot u$,}
\end{equation}
where $a \cdot  X = \{a \cdot v: v\in X\}$.

Note that $\vdash$ is action-invariant if and only if the associated consequence operator $\cns$ satisfies the condition 
\begin{equation}\label{subinv}
a \cdot \cns(X) \subseteq \cns(a \cdot X).
\end{equation} 
By extension, we call \emph{action-invariant} any consequence operator $\cns$ that satisfies the preceding condition.

For most consequence relations arising in logic, the assertions in the set $S$ are constructed in some way from elements of the term algebra. A \emph{propositional language} is a pair $\lang = \la L, \nu \ra$ consisting of a set $L$ and a map $\nu: L \lto \N_0$. The elements of $L$ are called \emph{connectives}, and the image of a connective under $\nu$ is called its \emph{arity}; nullary connectives are most often called \emph{constant symbols} or simply  \emph{constants}.

Given a propositional language $\lang$ and a denumerable set $\V = \{x_n \mid n \in \N\}$ of \emph{propositional variables}, the $\lang$-formulas are strings of connectives defined recursively by means of the following conditions:
\begin{enumerate}[(F1)]
\item every propositional variable and every constant symbol is an $\lang$-formula;
\item if $f$ is a connective of arity $\nu(f) > 0$ and $\phi_1, \ldots, \phi_{\nu(f)}$ are $\lang$-formulas, then $f\phi_1\ldots\phi_{\nu(f)}$ -- usually denoted by $f(\phi_1, \ldots, \phi_{\nu(f)})$ -- is an $\lang$-formula;
\item all $\lang$-formulas are built by iterative applications of (F1) and (F2).
\end{enumerate}

We denote the set of all $\lang$-formulas by $\fml$. For any $\phi \in \fml$, we write $\phi = \phi[x_{i_1}, \ldots, x_{i_n}]$ whenever we wish to indicate that the variables of $\phi$ are among those in the set $\{x_{i_1}, \ldots, x_{i_n}\}$.

On the algebraic side, a language $\lang$ can be used to specify the fundamental operations of an algebra (or class of algebras). In this case, we often use the term \emph{signature} in place of \emph{language}. 
If $\lang$ is a language over $\V$, then by the inductive definition of formulas, 
\begin{equation*}
\Fml = \langle \fml, \lang^{\Fml} \rangle
\end{equation*}
becomes an $\lang$-algebra, that is, an algebra of signature $\lang$. We refer to $\Fml$ as the \emph{term algebra} of signature $\lang$,  and recall that it is the free algebra over $\V$ in the class of all algebras of signature $\lang$. 

The vast majority of research in abstract algebraic logic is concerned with consequence relations on \emph{formula structures}. Without being precise about the exact meaning of the term `formula structure', we may assume for our purposes that it is a subset of the set $\seq$ of all sequents of a given signature, in the sense described below. Formulas are the elements of the term algebra $\Fml$, while  \emph{equations} are ordered pairs of formulas $(\f,\psi)$, often written suggestively as $\f \eq \psi$. They are just the elements of the algebra $\EQ = \la \Eq, \lang^{\EQ} \ra =\Fml^2$. Given  non-negative integers $m, n$ (not both equal to zero), a \emph{sequent over $\lang$ of type $(m,n)$} is a pair $(\Gamma, \Delta)$, consisting  of a sequence  $\Gamma = (\phi_1, \dots , \phi_m)$ of $\lang$-formulas of length $m$, and a sequence $\Delta = (\psi_1, \dots, \psi_n)$ of $\lang$-formulas of length $n$. Instead of $(\Gamma, \Delta)$ we usually write $\phi_1, \dots , \phi_m \To \psi_1, \dots, \psi_n$ or $\Gamma \To \Delta$.

Throughout this work, by a ``set of sequents'' $S$ we always understand a subset of $\seq$ which is closed under type, namely, such that whenever a sequent of type $(m,n)$ is  in $S$, then all the sequents of type $(m,n)$ are in $S$. Note that formulas may be identified with all $(0,1)$-sequents, and equations with all $(1,1)$-sequents.

In this setting, the endomorphisms on the algebra of formulas $\Fml$ form a monoid $\Sfm = \la \sfm, \circ, \id_{\fml}\ra$ which has a natural action on $\fml$: $\sigma \cdot \phi =  \sigma(\phi) $, for all $\sigma \in \sfm$ and $\phi \in \fml$. We refer to $\Sfm$ as the monoid of substitutions, and remark that its action on $\fml$ has a natural extension to $\seq$. A consequence relation on a formula structure that is action-invariant with respect to substitutions will be called \emph{substitution-invariant}.\footnote{The term ``structural" is also widely used, however, it is somewhat misleading because this property bears no relationship to the concepts of structural rule or substructural logic.} 

\begin{definition}\label{dedsys}
A \emph{(propositional) deductive system} $\cat S = \la S, \vdash\ra$ is a pair consisting of a set $S$ of sequents over a propositional language $\lang$ and a substitu\-tion-invariant consequence relation over $S$ or, what amounts to the same, a substitution-invariant consequence operator on $\wp(S)$.
\end{definition}

The possibility for two deductive systems to be equipped with entailments which somehow look alike, even if presented under different guises, is described by the following definition. 

\begin{definition}\label{intlang}
Let $\AA$ be a monoid, $S$ and $T$ two $\AA$-sets, and $\vdash_S$ and $\vdash_T$ action-invariant consequence relations on $S$ and $T$ respectively.
\begin{enumerate}[(i)]
\item A map $\iota: S \lto \wp(T)$ is said to be \emph{action-invariant} if $\iota(a \cdot_S x) = a \cdot_T \iota(x)$ for all $a \in A$ and $x \in S$.
\item An action invariant map $\iota: S \lto \wp(T)$ is called an \emph{interpretation} of $\vdash_S$ in $\vdash_T$ if, for all $X \cup \{u\} \subseteq S$, it satisfies
\begin{equation}\label{neqint}
X \vdash_S u \quad \textrm{implies} \quad \iota[X] \vdash_T \iota(u).
\end{equation}
\item An action invariant map $\iota: S \lto \wp(T)$ is called a \emph{representation}, or a \emph{conservative interpretation}, of $\vdash_S$ in $\vdash_T$ if, for all $X \cup \{u\} \subseteq S$, it satisfies 
\begin{equation}\label{eqint}
X \vdash_S u \quad \textrm{if and only if} \quad \iota[X] \vdash_T \iota(u).
\end{equation}
\item Two representations $\iota: S \lto \wp(T)$ and $\iota': T \lto \wp(S)$ are said to form an \emph{equivalence} if, for all $v \in T$,
\begin{equation}\label{eqeq}
v \dashv\vdash_T \iota[\iota'(v)].
\end{equation}
In this case, we say that $\vdash_S$ and $\vdash_T$ are \emph{equivalent}.\footnote{It can be easily shown that (iv) can be equivalently formulated by substituting (\ref{eqeq}) with
\begin{equation}\label{eqeq'}
u \dashv\vdash_S \iota'[\iota(u)].
\end{equation}}
\item A \emph{weak} interpretation (respectively: representation) of $\vdash_S$ in $\vdash_T$ is a map $\iota: S \lto \wp(T)$ which satisfies (\ref{neqint}) (resp.: (\ref{eqint})) but is not necessarily action invariant. Two weak representations $\iota$ and $\iota'$ satisfying (\ref{eqeq}) are said to form a \emph{similarity}; in this case, $\vdash_S$ and $\vdash_T$ are called \emph{similar}. 
\end{enumerate}
\end{definition}

Important examples of equivalent consequence relations involve algebraizable consequence relations in the sense of Blok and Pigozzi \cite{blokpigozzi}. More specifically, if $\vdash$ is an algebraizable consequence relation on $\fml$ with equivalent algebraic semantics a class $\mathcal{K}$ of algebras, then the consequence relation $\models$ on $\Eq$ arising from $\mathcal{K}$ is equivalent to $\vdash$ (see, for example, \cite{blokjonsson2,galtsi}).

We remark that Definition \ref{intlang} can be easily reformulated in terms of consequence operators by using Lemma \ref{oprel}.

\section{Deductive systems as categories}
\label{catsec}

In the present section, which may be viewed as parenthetic to the subsequent discussion, we justify the intuition that equivalence of two consequence relations is intimately related to categorical equivalence. For any notion or result on category theory not explicitly reported here, we refer the reader to \cite{cats}.

A \emph{preorder} on a set $X$ is a reflexive and transitive binary relation $R$ on $X$; a preordered set $\la X, R \ra$ may be thought of as a category whose objects are the members of the set and whose morphisms are the pairs $(x,y)$ such that $x R y$. In what follows, we denote the associated equivalence relation by $\approx_{R}$: $x\approx_{R}y$ if and only if $xRy$ and $yRx$. 

Note that given two such categories $\la X, R \ra$ and $\la Y, S \ra$, a map $F : X\lto Y$ is a functor if and only if it is relation preserving: $aRb$ implies $F(a)SF(b)$. Of particular interest to us is the situation when  both $X$ and $Y$ are $\AA$-sets for a given monoid $\AA$. In this case, we use the term \emph{action-invariant} for a functor 
$F : \la X, R \ra\lto \la Y, S \ra$ that preserves scalar multiplication.

If $\cat S = \la S, \vdash \ra$ is a deductive system, $\vdash$ induces a preorder on $\wp(S)$, also denoted by $\vdash$; namely, for all $X, Y\subseteq S$,  $X \vdash Y$ whenever $X \vdash y$ for all $y \in Y$. In what follows, we denote the category $\la \wp(S), \vdash \ra$ by $\ov{\cat S}$.

Since $X \supseteq Y$ implies $X \vdash Y$, it is immediate that such a category has both an initial object $S$ and a terminal object $\varnothing$. It is worth mentioning, in the concrete case of deductive systems over a propositional language,  that the set $\T_{\cat S}$ of theorems  of $\cat S$ is a terminal object of $\ov{\cat S}$; therefore $\T_{\cat S} \cong \varnothing$ in $\ov{\cat S}$.

A category is said to be \emph{thin} if, for any two objects $A$ and $B$, there exists at most one morphism from $A$ to $B$; it is well-known that thin categories are, up to isomorphisms, precisely the preordered classes. Recall that a functor $F : \mathcal C \lto \mathcal D$ between two categories is called 
\begin{itemize}
\item \emph{faithful} provided that the hom-set restrictions are injective,
\item \emph{full} if the hom-set restrictions are surjective, and
\item \emph{isomorphism-dense} if for any $\cat D$-object $D$ there exists a $\cat C$-object $C$ such that $FC$ and $D$ are isomorphic.
\end{itemize}
Any functor whose domain is a thin category is obviously faithful.

An \emph{equivalence} is a full, faithful and isomorphism-dense functor. Equivalently, a functor $F : \mathcal C \lto \mathcal D$ is an equivalence if there exist an ``inverse'' functor $G : \mathcal D \lto \mathcal C$ and two natural isomorphisms $\eta : \ID_{\mathcal C}\lto GF$ and $\varepsilon : \ID_{\mathcal D}\lto FG$. Here, $\ID_{\mathcal C}$ and $\ID_{\mathcal D}$ denote the identity functors on $\mathcal C$ and $\mathcal D$, respectively. It is clear that  if $F$ is an equivalence, then so is the companion functor $G$. Specializing to the case of preordered sets, a functor $F : \la X, R \ra\lto \la Y, S \ra$ is a categorical equivalence provided there exists a functor $G : \la Y, S \ra \lto \la X, R \ra$ such that $x\approx_{R}GF(x)$ and $y\approx_{S}FG(y)$, for all $x\in X$ and $y\in Y$.

The next result shows that Definition \ref{intlang} is intuitively justified also from the categorical viewpoint.

\begin{theorem}\label{catint}
Let $\AA$ be a monoid, $S$ and $T$ two $\AA$-sets, $\la S, \vdash_S\ra$ and $\la T, \vdash_T\ra$ two action-invariant deductive systems over $S$ and $T$ respectively, $f: S \lto \wp(T)$ a map and $F: X \in \wp(S) \lmapsto f[X] \in \wp(T)$. Then the following hold:
\begin{enumerate}[(i)]
\item $f$ is an interpretation of $\la S, \vdash_S\ra$ into $\la T, \vdash_T\ra$ if and only if $F$ is an action-invariant faithful functor between the corresponding categories $\ov{\cat S}$ and $\ov{\cat T}$;
\item $f$ is a conservative interpretation if and only if $F$ is an action-invariant full and faithful functor;
\item if $g: T \lto \wp(S)$ is a map and $G: Y \in \wp(T) \lmapsto G[Y] \in \wp(S)$, $f$ and $g$ form an equivalence if and only if $F$ and $G$ are action-invariant and form a categorical equivalence.
\end{enumerate}
\end{theorem}
\begin{proof}
The fact that the map $f$ is action-invariant if and only if so is $F$ is trivial.

Now, with reference to Definition \ref{intlang}, $f$ is an interpretation if and only if any morphism $X \vdash_S Y$ is mapped by $F$ to a morphism $FX \vdash_T FY$, that is, $F$ is a (faithful) functor. The converse implication in (\ref{eqint}) can be reformulated in the categorical setting as ``if there exists a morphism $FX \vdash_T FY$ then there exists a morphism $X \vdash_S Y$'', and the latter holds if and only if $F$ is full. 

Last, the equivalent conditions (\ref{eqeq}) and (\ref{eqeq'}) hold if and only if, respectively, $Y \approx_T FG(Y)$ and $X \approx_S GF(X)$, i.e., if and only if $F$ is a categorical equivalence with inverse $G$. 
\end{proof}

The following result on weak interpretations readily follows from Theorem \ref{catint}.
\begin{corollary}
Let $\la S, \vdash_S\ra$ and $\la T, \vdash_T\ra$ be two deductive systems over $S$ and $T$ respectively, $f: S \lto \wp(T)$ a map and $F: X \in \wp(S) \lmapsto f[X] \in \wp(T)$. Then the following hold:
\begin{enumerate}[(i)]
\item $f$ is a weak interpretation of $\la S, \vdash_S\ra$ into $\la T, \vdash_T\ra$ if and only if $F$ is a faithful functor between the corresponding categories $\ov{\cat S}$ and $\ov{\cat T}$;
\item $f$ is a weak representation if and only if $F$ is a full and faithful functor;
\item if $g: T \lto \wp(S)$ is a map and $G: Y \in \wp(T) \lmapsto G[Y] \in \wp(S)$, $f$ and $g$ form a similarity if and only if $F$ and $G$ form a categorical equivalence.
\end{enumerate}
\end{corollary}

\section{Translations}
\label{transsec}

Taking a closer look at Definition \ref{intlang}, and considering the concrete cases of propositional deductive systems, it should appear evident that such a definition is not satisfactory if we deal with systems defined on different underlying languages. Indeed, in this case it is able to describe only weak interpretations, while a notion of action-invariance is not even defined.

As we anticipated, aim of this paper is precisely to extend the algebraic and categorical approach of Galatos and Tsinakis \cite{galtsi} to such a situation; in order to do that, it is necessary to understand what does ``action-invariant'' mean in this case, namely, when the actions come from different monoids. So, in this section we shall define language translations, prove some results about them, and extend Definition \ref{intlang} to the most general case. Such results (in particular Lemma \ref{qinq'} and Theorem \ref{transchar}) will pave the way to the characterizations of Section \ref{interabs}. 

Let $\lang = \la L, \nu \ra$ be a propositional language. If $n \in \N_0$ and $f: \fml^n \lto \fml$ is a map, $f$ is called a \emph{derived operation on $\fml$} if there exists a formula $\phi_f = \phi_f[x_1, \ldots, x_n] \in \fml$ in the $n$ variables $x_1, \ldots, x_n$ such that $f(\psi_1, \ldots, \psi_n) = \phi_f[x_1/\psi_1, \ldots, x_n/\psi_n]$, for all $\psi_1, \ldots, \psi_n \in \fml$. In particular, if $n = 0$, $f$ is a \emph{derived constant}, i.e., a formula in $\fml$ containing only constants and no variables.

We use derived operations to define the notion of a language translation.

\begin{definition}\label{translation}
Let $\lang = \la L, \nu \ra$ and $\lang' = \la L', \nu' \ra$ be two propositional languages. Assume that for each connective $f \in L$ there exists a derived operation $f'$ on $\fmll$ of arity $\nu(f)$. If we denote by $\lang^{\Fmll}$ the set of such operations, the structure $\mathbf{Fm}'_{\lang} = \la \fmll, \lang^{\Fmll} \ra$ is an $\lang$-algebra. In this case, a map $\tau: \fml \lto \fmll$ is called a \emph{language translation} of $\lang$ into $\lang'$ if
\begin{enumerate}[(i)]
\item $\tau^{-1}(x) = \{x\}$ for any variable $x$,
\item $\tau$ is an $\lang$-homomorphism, that is
$$\tau(f(\phi_1, \ldots, \phi_{\nu(f)})) = f'(\tau(\phi_1), \ldots, \tau(\phi_{\nu(f)})),$$
for all $f \in L$ and $\phi_1, \ldots, \phi_{\nu(f)} \in \fml$.
\end{enumerate}\end{definition}

\begin{lemma}\label{sigmainend}
Let $\tau$ be a language translation of $\lang = \la L, \nu \ra$  into $\lang' = \la L', \nu' \ra$. The monoid of substitutions $\Sfmm$ (over the language $\lang'$) of $\fmll$ is a submonoid of the endomorphism monoid $\mathbf{End}_\lang(\mathbf{Fm}'_{\lang})$ of the $\lang$-algebra $\mathbf{Fm}'_{\lang} = \la \fmll, \lang^{\Fmll} \ra$.
\end{lemma}
\begin{proof}
The inclusion $\sfmm \subseteq \operatorname{End}_\lang(\mathbf{Fm}'_{\lang})$ comes easily from the fact that the operations of $\mathbf{Fm}'_{\lang}$ are derived operations on $\fmll$, so they are preserved by any $\lang'$-substitution of $\fmll$. Thus $\Sfmm$ is  a submonoid of $\mathbf{End}_\lang(\mathbf{Fm}'_{\lang})$, for it contains the identity map and is closed under composition.
\end{proof}

We note that the reverse inclusion in the preceding lemma does not hold in general. For example, let $g_1$ and $g_2$ be two $n$-ary connectives of $\lang'$ not involved in any of the formulas that define the operations in $\lang^{\Fmll}$ and, for all $\phi \in \fmll$, let $h(\phi)$ be the formula obtained from $\phi$ by substituting each occurrence (if any) of $g_1$ by $g_2$. Then $h: \fmll \lto \fmll$ is an $\lang$-endomorphism of $\mathbf{Fm}'_{\lang}$ that is not a substitution of $\lang'$. 

\begin{lemma}\label{qinq'}
Let $\tau$ be a language translation of $\lang = \la L, \nu \ra$  into $\lang' = \la L', \nu' \ra$. The following hold:
\begin{enumerate}[(i)]
\item $\tau$ induces a monoid homomorphism $\ov \tau : \Sfm \lto \Sfmm$. More concretely, for each $\sigma \in \sfm$, let $\sigma'$ be the substitution uniquely determined by the map $\tau \circ \sigma_{\restr \V} \in \fmll^\V$. Then $\ov \tau$ is defined by $\ov \tau(\sigma)=\sigma'.$
\item $\ov \tau$ is injective (resp.: surjective) if and only if $\tau$ is.
\item $\tau$ commutes with the substitutions in $\sfm$ in the following sense: $\tau(\sigma(\phi)) = \ov\tau(\sigma)(\tau(\phi))$ for all $\sigma \in \sfm$ and $\phi \in \fml$.
\end{enumerate}
\end{lemma}
\begin{proof}
\begin{enumerate}[(i)]
\item Obviously $\ov \tau(\id_{\fml}) = \id_{\fmll}$. Now let us show that $\ov \tau(\sigma_2 \circ \sigma_1) = \ov \tau(\sigma_2) \circ \ov \tau(\sigma_1)$ for all $\sigma_1, \sigma_2 \in \sfm$. For every formula $\phi[x_1, \ldots, x_n] \in \fml$ in the variables $x_1, \ldots, x_n$, set $\tau(\phi)=\phi'[x_1, \ldots, x_n]$. Note that $\phi'$ is obtained from $\phi$ by replacing each connective $f\in \lang$ which occurs in $\phi$ by $f'\in \lang^{\Fmll}$.

Let $x$ be an arbitrary variable, and let $\sigma_1, \sigma_2\in \Sfm$. Then $\sigma_1(x) = \phi[x_1, \ldots, x_n]$ and $\sigma_2(x_i) = \psi_i[x_{i1}, \ldots, x_{ik_i}]$ ($i = 1, \ldots, n$), for suitable formulas  $\phi, \psi_1, \ldots, \psi_n \in \fml$. Consider the formula in $\sum_{i=1}^n k_i$ variables
$$\phi\langle\psi_1,\ldots, \psi_n\rangle[x_{11}, \ldots, x_{nk_n}]=\phi[\psi_1[x_{11}, \ldots, x_{1k_1}], \ldots, \psi_n[x_{n1}, \ldots, x_{nk_n}]].$$
Then the computation below establishes (i).
$$\begin{array}{l}
  \ov \tau(\sigma_2 \circ \sigma_1)(x) \\
= \tau((\sigma_2 \circ \sigma_1)(x)) \\
= \tau(\sigma_2(\phi[x_1, \ldots, x_n])) \\
= \tau(\phi[\sigma_2(x_1), \ldots, \sigma_2(x_n)]) \\
= \tau(\phi\langle\psi_1,\ldots, \psi_n\rangle[x_{11}, \ldots, x_{nk_n}]) \\
= (\phi\langle\psi_1,\ldots, \psi_n\rangle)'[x_{11}, \ldots, x_{nk_n}] \\
= \phi'[\psi'_1[x_{11}, \ldots, x_{1k_1}], \ldots, \psi'_n[x_{n1}, \ldots, x_{nk_n}]] \\
= \phi'[\tau(\psi_1[x_{11}, \ldots, x_{1k_1}]), \ldots, \tau(\psi_n[x_{n1}, \ldots, x_{nk_n}])] \\
= \phi'[\tau(\sigma_{2}(x_1)), \ldots, \tau(\sigma_{2}(x_n))] \\
= \phi'[\ov \tau(\sigma_{2})(x_1), \ldots,\ov \tau(\sigma_{2})(x_n)] \\
= \ov \tau(\sigma_2)(\phi'[x_1, \ldots, x_n]) \\
= \ov \tau(\sigma_2)(\tau(\phi[x_1, \ldots, x_n])) \\
= \ov \tau(\sigma_2)(\ov \tau(\sigma_1)(x)) \\
= (\ov \tau(\sigma_2) \circ \ov \tau(\sigma_1))(x).
\end{array}$$

\item If $\tau$ is injective (resp.: surjective), then $\ov \tau$ is obviously injective (resp.: surjective) too.

Conversely, if $\tau$ is not injective, then there exist two different formulas $\phi, \psi \in \fml$ such that $\tau(\phi) = \tau(\psi)$. Therefore, if we consider the two substitutions $\sigma_\phi$ and $\sigma_\psi$ that send a variable $x$ respectively to $\phi$ and $\psi$, and fix the elements  of $\V \setminus \{x\}$, we have two different substitutions whose images under $\ov \tau$ coincide. Hence $\ov \tau$ is injective if and only if $\tau$ is injective. The fact that $\ov\tau$ surjective implies $\tau$ surjective can be easily proved with an analogous argument.
\item Last property follows immediately from the definition of $\ov \tau$.
\end{enumerate}
\end{proof}

Now that we know that any language translation induces a homomorphism between the substitution monoids of the same languages, we provide, as a next step,  a characterization of homomorphisms between substitution monoids that are induced by language translations.

We first observe that, for any language $\lang$, the set $V=\{\sigma \in \sfm \mid \sigma[\V] \subseteq \V\}$ is the universe of a submonoid $\VV$ of $\Sfm$. Furthermore, if $\VV$ and $\VV'$ are two such submonoids corresponding to the languages $\lang$ and $\lang'$, then $\VV\cong\VV'$. In the sequel, we identify all these monoids and denote them by $\VV$.

Let us also recall a notion from the theory of semigroups: if $B$ is a subset of a monoid $\AA$, an element $a \in A$ is called \emph{right zero for $B$} if $ba = a$ for all $b \in B$. The following result is trivial.

\begin{lemma}\label{abs}
Let $\AA$ and $\AA'$ be monoids, $B \subseteq A$ and $g: \AA \lto \AA'$ a monoid homomorphism. If $a \in A$ is a right zero for $B$, then $g(a)$ is a right zero for $g[B]$ in $\AA'$.
\end{lemma}

\begin{theorem}\label{transchar}
Let $h: \Sfm \lto \Sfmm$ be a monoid homomorphism. Then $h$ is induced by a language translation of $\lang$ into $\lang'$ if and only if it satisfies the following conditions:
\begin{enumerate}[(i)]
\item if $\rho$ is an idempotent element of $\Sfmm$, then $h^{-1}(\rho)$ is either empty or is comprised of idempotent elements of $\Sfm$;
\item $h^{-1}(\sigma) = \{\sigma\}$, for all $\sigma \in V$. 
\end{enumerate}
\end{theorem}
\begin{proof}
If $h = \ov\tau$ for some translation $\tau: \fml \lto \fmll$, then it satisfies (ii) by Definition~\ref{translation}. To prevent any confusion, we remind the reader that we consider $\VV$ as a common submonoid of $\Sfm$ and $\Sfmm$, and therefore we use the same symbol for two substitutions in $\Sfm$ and $\Sfmm$ that belong to $\VV$.

In order to prove (i), let $\rho$ be an idempotent element of $\sfmm$, $\sigma \in \ov\tau^{-1}(\rho)$, and $\phi[x_1, \ldots, x_n] \in \sigma[\fml]$. It must be shown that $\sigma(\phi)=\phi$. Since $\tau$ is a translation and $\ov\tau(\sigma)_{\restr \V} = \tau \circ \sigma_{\restr \V} = \rho_{\restr \V}$ (Lemma \ref{qinq'}(i)), there exists $\phi'[x_1, \ldots, x_n] \in \rho[\fmll]$ such that that $\phi' = \tau(\phi)$. The idempotency of $\rho$ yields the equality $\rho(\phi') = \phi'$, which implies that $\rho(x_{j}) = x_{j}$ for all $j \leq n$; that is, all variables occurring in $\phi'$ are fixed by $\rho$.  Thus, invoking again the equalities $\ov\tau(\sigma)_{\restr \V} = \tau \circ \sigma_{\restr \V} = \rho_{\restr \V}$,  we have that $\sigma$ fixes all variables $x_1, \ldots, x_n$. It follows that $\sigma(\phi)=\phi$, and so $\sigma$ is idempotent.

Conversely, let us assume that $h$ satisfies (i) and (ii) and fix $x \in \V$. For every $\phi\in \Fml$, let $\sigma_\phi\in\sfm$ be a substitution such that $\sigma_\phi(x)=\phi$. Define $\tau: \Fml \lto \Fmll$ by $\tau(\phi) = h(\sigma_\phi)(x)$.

Note that $\tau$ is well defined. Indeed, let $\sigma_\phi, \sigma'_\phi\in \sfm$ such that $\sigma_\phi(x) = \sigma'_\phi(x) = \phi$ and let $\kappa_x$ be the substitution that sends all variables to $x$. Then $\sigma_\phi(x) = \sigma'_\phi(x) = \phi$ is equivalent to $\sigma_\phi\circ\kappa_x = \sigma'_\phi\circ\kappa_x$ in $\sfm$. It follows that $h(\sigma_\phi) \circ \kappa_x = h(\sigma_\phi\circ\kappa_x) = h(\sigma'_\phi\circ\kappa_x) = h(\sigma'_\phi)\circ\kappa_x$, which yields $h(\sigma_\phi)(x) = h(\sigma^\prime_\phi)(x)$. We shall prove that $\tau$ is a language translation and $h = \ov\tau$.

As a first step, assume that $c$ is a constant of $\lang$, and let $\sigma_c$ be a substitution that sends $x$ to $c$. For any substitution $\sigma \in \sfm$, $(\sigma \circ\sigma_c)(x) = \sigma_c(x)$; we assume, by contradiction, that the formula $h(\sigma_c)(x)$ contains a variable $y$. Let $\alpha$ be the substitution that sends $y$ to another variable $z \neq y$ and fix all the other variables. Then $\alpha \in V$ which implies $h(\alpha) = \alpha$; on the other hand $\tau((\alpha \circ\sigma_c)(x)) = (\alpha \circ h(\sigma_c))(x) \neq h(\sigma_c)(x) = \tau(\sigma_c(x))$, and this is absurd since $(\alpha\circ \sigma_c)(x) = \sigma_c(x)$. Therefore $h(\sigma_c)(x)$ cannot contain variables and $\lang'$ must have a definable constant.\footnote{We observe explicitly that, in order to define constants, a language must have at least a primitive constant.}

The case of unary connectives could be treated within the general case; nonetheless we prove it separately in order to give the reader a better clue of the argument.

Let $f$ be a unary connective of $\lang$ and $\sigma_f$ a substitution that sends $x$ to $f(x)$. Now let us consider the following subset of $V$:
$$B = \{\sigma \in V \mid \sigma(x) = x \textrm{ and } \sigma(y) \neq y, \forall y \in \V \setminus \{x\}\}.$$
Such a set is easily seen to be non-empty and it is clear that $\sigma_f\circ \kappa_x$ (i.e., the substitution that sends all variables to $f(x)$) is a right zero for $B$, hence $\sigma\circ \sigma_f\circ\kappa_x = \sigma_f\circ\kappa_x$ for all $\sigma \in B$. Then, if $B' = h[B]$, by Lemma~\ref{abs}, $h(\sigma_f) \circ \kappa_x$ is a right zero for $B'$, which means essentially that $\tau(\sigma_f(x))$ contains at most the unique variable $x$. On the other hand, if $\tau(\sigma_f(x))$ is a constant, then $h(\sigma_f)\circ \kappa_x$ is an idempotent element of $\sfmm$ while, for any $y \in \V$, $(\sigma_f\circ \kappa_x\circ \sigma_f\circ \kappa_x)(y) = f(f(x)) \neq f(x) = (\sigma_f\circ \kappa_x)(y)$, i.e., $\sigma_f\circ \kappa_x$ is not idempotent. But this is impossible by hypothesis (i); so $\tau(\sigma_f(x))$ cannot be a constant, and therefore it is a formula in the single variable $x$.

Now let $f$ be a connective of arity $n > 1$ and $\sigma_f$ be the substitution that sends $x$ to $f(x_1, \ldots, x_n)$, with $x, x_1, \ldots, x_n$ distinct variables, and acts like the identity on $\V \setminus \{x\}$; let also $X = \V \setminus \{x_1, \ldots, x_n\}$ and consider the subset $B$ of $V$ defined by
$$B = \{\sigma \in V \mid \sigma(x_i) = x_i, \forall i = 1, \ldots, n \textrm{ and } \sigma(y) \neq y, \forall y \in X\}.$$
$B$ is clearly non-empty and $\sigma_f\circ \kappa_x$ (i.e., the substitution sending all variables to $f(x_1, \ldots, x_n)$) is a right zero for $B$, hence $\sigma\circ \sigma_f\circ \kappa_x = \sigma_f\circ \kappa_x$ for all $\sigma \in B$. As in the case of unary connectives, if $B' = h[B]$, by Lemma~\ref{abs}, $h(\sigma_f)\circ \kappa_x$ is a right zero for $B'$, which means that $\tau(\sigma_f(x))$ contains at most the variables $x_1, \ldots, x_n$. Assuming that there exists $i \leq n$ such that $x_i$ is not in $\tau(\sigma_f(x))$, we can consider the substitution $\alpha$ that sends $x_i$ to $x$ and acts as the identity on $\V \setminus \{x_i\}$. Then $h(\sigma_f\circ \alpha\circ \sigma_f\circ \kappa_x)$ is easily seen to be idempotent while $\sigma_f\circ \alpha\circ \sigma_f\circ \kappa_x$ is not. Again, this is impossible by (i), therefore $\tau(\sigma_f(x))$ contains precisely the variables $x_1, \ldots, x_n$.

Now we must prove that $\tau$ is a language translation and $h = \ov\tau$. Condition (i) of Definition~\ref{translation} is an obvious consequence of hypotheses (i) and (ii): for any variable $y$, $\tau(y) = h(\kappa_y)(x) = \kappa_y(x) = y$. Regarding Definition~\ref{translation} (ii), we observe that, for any connective $f \in \lang$, $\tau(f(x_1, \ldots, x_n))$ is a formula $f'(x_1, \ldots, x_n)$ in the variables $x_1, \ldots, x_n$. So, let $\phi_1, \ldots, \phi_n \in \fml$ and $\sigma \in \sfm$ be the substitution that sends $x_i$ to $\phi_i$, for all $i = 1, \ldots, n$; we have $\tau(\phi_i) = \tau(\sigma\circ\kappa_{x_i}(x)) = (h(\sigma)\circ\kappa_{x_i})(x)$, for all $i \leq n$, whence
$$\begin{array}{l}
\tau(f(\phi_1, \ldots, \phi_n)) = h(\sigma\circ \sigma_f)(x) = (h(\sigma)\circ h(\sigma_f))(x) =\\
h(\sigma)(f'(x_1, \ldots, x_n)) = f'(h(\sigma)(x_1), \ldots, h(\sigma)(x_n)) = \\
f'((h(\sigma)\circ\kappa_{x_1})(x), \ldots, (h(\sigma)\circ\kappa_{x_n})(x)) = f'(\tau(\phi_1), \ldots, \tau(\phi_n)).
\end{array}$$

Last we must show that $h = \ov\tau$. Since substitutions are completely and univocally determined by their restriction to $\V$, we can represent any $\sigma \in \sfm$ by the family $\{\sigma\circ \kappa_{x_i}\circ \kappa_x\}_{i \in \N} = \{\sigma(x_i)\}_{i \in \N}$. Then $h(\sigma)$ is completely determined by
$$\{h(\sigma)(x_i)\}_{i \in \N} = \{h(\sigma\circ \kappa_{x_i}\circ \kappa_x)\}_{i \in \N} = \{h(\sigma)\circ \kappa_{x_i}\circ \kappa_x\}_{i \in \N} = \tau \circ \sigma_{\restr \V},$$
that is, $h = \ov\tau$. The theorem is proved. 
\end{proof}

\begin{theorem}\label{retraction}
Let $\tau: \lang \lto \lang'$ be a translation. Then $\ov\tau: \Sfm \lto \Sfmm$ is a monoid retraction if and only if $\tau$ is onto.
\end{theorem}
\begin{proof}
By Lemma \ref{qinq'}(ii), $\tau$ is onto if and only if $\ov\tau$ is onto, hence we need to prove that $\ov\tau$ is a retraction if and only if it is onto. One implication is trivial, namely, a retraction in a concrete category has a surjective underlying map.

Now assume $\tau$ to be surjective. By Definition \ref{translation}, for any connective $f$ of $\lang$, $\tau(f(x_1, \ldots, x_n))$ contains at least one connective of $\lang'$. Therefore, for any formula $\phi \in \fml$, the number of connectives in $\tau(\phi)$ is necessarily greater than or equal to the number of connectives in $\phi$. So, since $\tau$ is onto, for any connective $f' \in \lang'$ there exists a connective $f \in \lang$ having the same arity of $f'$ and such that (with an abuse of notation) $\tau(f) = f'$.

Then we can define\footnote{Note that the Axiom of Choice here is needed only in the case where $\lang'$ has infinitely many connectives.} a translation $\tau': \lang' \lto \lang$ by mapping each $f' \in \lang'$ in such an $f \in \lang$ and extending such a map recursively according to Definition \ref{translation}. It is immediate to verify that $\tau \circ \tau' = \id_{\fmll}$ and therefore $\ov\tau \circ \ov\tau' = \id_{\sfmm}$.
\end{proof}

We close this section by extending the Definition \ref{intlang}(i--iv) to the case of consequence relations that are action-invariant w.r.t. actions from different monoids and, therefore, to the case of systems with different underlying languages by virtue of Lemma \ref{qinq'}.

\begin{definition}\label{ints}
Let $\mathbf A$ and $\mathbf B$ be monoids, $S$ an $\mathbf A$-set, $T$ a $\mathbf B$-set, and $\vdash_S$ and $\vdash_T$ two action-invariant consquence relations on $S$ and $T$ respectively.

A map $\iota: S \lto \wp(T)$ is said to be \emph{action-invariant via the monoid homomorphism $h$} (\emph{$h$-action-invariant}, for short) if there exists a monoid homomorphism $h: \mathbf A \lto \mathbf B$ such that $\iota(a \cdot_S x) = h(a) \cdot_T \iota(x)$ for all $a \in A$ and $x \in S$.

A map $\iota: S \lto \wp(T)$ is called an \emph{interpretation} (respectively: a \emph{representation}) \emph{via $h$} of $\vdash_S$ in $\vdash_T$ if it is a weak interpretation (resp.: weak representation) and is $h$-action-invariant for some monoid homomorphism $h$.

A similarity given by two weak representations $\iota: S \lto \wp(T)$ and $\iota': T \lto \wp(S)$ is called an \emph{equivalence via $h$ and $k$} if there exist two monoid homomorphisms $h: \mathbf A \lto \mathbf B$ and $k: \mathbf B \lto \mathbf A$ such that $\iota$ is $h$-action-invariant and $\iota'$ is $k$-action invariant.
\end{definition}

\section{Some preliminaries on quantales and quantale modules}
\label{ressec}

It is known that an efficient abstract algebraic representation for a propositional deductive system is that of a quantale module \cite{galtsi,thesis}. Indeed, as we are going to see, in such a representation consequence relations are easily described as congruences and consequence operators are precisely the nuclei on quantale modules, i.e. special closure operators whose definition will be recalled in this section. The details of this representation shall be recalled in the next section; here, as a preliminary step, we recall the notions and results on quantales and their modules that we will need.

If $\la X, \leq \ra$ and $\la Y, \leq \ra$ are two posets, a map $f: X \lto Y$ is said to be \emph{residuated} provided there exists a map $g: Y \lto X$ such that, for all $x \in X$ and for all $y \in Y$, $f(x) \leq y$ if and only if $x \leq g(y)$. It is immediate to verify that the map $g$ is uniquely determined; we call it the \emph{residual map} or the \emph{residuum} of $f$, and denote it by $f_*$. The pair $(f, f_*)$ is said to be \emph{adjoint} or to form a \emph{Galois connection between $X$ and $Y$}.

We also recall that the category $\SL$ of \emph{sup-lattices} have complete lattices as objects and residuated maps (or, that is the same, maps preserving arbitrary joins) as morphisms. For any set $S$, the free sup-lattice over $S$ is $\la \wp(S), \bigcup\ra$ equipped with the singleton map, that is, with the map which sends each $x \in S$ to its corresponding singleton in $\wp(S)$.

A map $\g: X \lto X$ is called a \emph{closure operator} if it is order preserving, extensive and idempotent. It is well-known that a map $f: X \lto Y$ is residuated if and only if it preserves all existing joins. Moreover, its residuum $f_*$ preserves all existing meets and $\g = f_* \circ f$ is a closure operator on $X$.

\begin{definition}\label{quantale}
A \emph{(unital) quantale} is a monoid in the category of sup-lattices. In other words, a quantale is an algebraic structure $\QQ = \la Q, \bigvee, \cdot, 1 \ra$ such that
\begin{enumerate}[(Q1)]
\item $\la Q, \bigvee \ra$ is a sup-lattice,
\item $\la Q, \cdot, 1 \ra$ is a monoid,
\item the multiplication distributes over arbitrary joins both from the left and from the right.
\end{enumerate}
$\QQ$ is said to be \emph{commutative} if so is the multiplication and \emph{integral} if $1 = \bigvee Q$.
\end{definition}

A homomorphism between two quantales $\QQ$ and $\RR$ is a map $f: Q \lto R$ that preserves arbitrary joins and the monoid structure. So, in particular, a quantale homomorphisms is a residuated monoid homomorphism and the map $\g \bydef f_* \circ f: Q \lto Q$ is a closure operator with the following additional property: $\g(a) \g(b) \leq \g(a b)$. An operator with these features is also called \emph{quantic nucleus} and induces a quantale structure on its image by setting, for all $a, b \in \g[Q]$ and $X \subseteq \g[Q]$,
$${}^\g\bigvee X = \g\left({}^\QQ\bigvee X\right) \quad \textrm{ and } \quad a \cdot_\g b = \g(ab).$$
We refer the reader to \cite{krupas,rose} for further information on quantales.

\begin{definition}\label{modules}
Let $\QQ$ be a quantale. A (left) \emph{$\QQ$-module} $\MM$, or a \emph{module over $\QQ$}, is a sup-lattice $\la M, \bigvee \ra$ endowed with an external binary operation, called \emph{scalar multiplication}, $\cdot: (q,x) \in Q \times M \lmapsto q \cdot x \in M$, such that the following conditions hold:
\begin{enumerate}[(M1)]
\item $(ab) \cdot x = a \cdot (b \cdot x)$, for all $a, b \in Q$ and $x \in M$;
\item the external product distributes over arbitrary joins both in $\QQ$ and $\MM$, i.e., for all $A \cup \{a\} \subseteq Q$ and for all $X \cup \{x\} \subseteq M$
\begin{itemize}
\item[-] $a \cdot \bigvee X = \bigvee \{a \cdot y \mid y \in X\}$,
\item[-] $\left(\bigvee A\right) \cdot x = \bigvee \{b \cdot x \mid b \in A\}$,
\end{itemize}
\item $1 \cdot x = x$.
\end{enumerate}
\end{definition}

Condition (M2) can be expressed, equivalently, as follows:
\begin{enumerate}
\item[(M2')]The scalar multiplication is residuated with respect to the lattice order in $M$, i.e.,
\begin{itemize}
\item[-] for all $a \in Q$, the map $a \cdot \<: x \in M \lmapsto a \cdot x \in M$ is residuated,
\item[-] for all $x \in M$, the map $\< \cdot x: a \in Q \lmapsto a \cdot x \in M$ is residuated.
\end{itemize}
\end{enumerate}

Then (M2') defines another external operation over $M$, with coefficients in $Q$, and a map from $M \times M$ to $Q$:
$$\begin{array}{c}
\ust: (a,x) \in Q \times M \lmapsto a \ust x = (a \cdot \<)_*(x) \in M,\\
 \\
\ost: (x,y) \in M \times M \lmapsto x \ost y = (\< \cdot x)_*(y) \in Q.
\end{array}$$ 
Let $\QQ$ be a quantale and $X$ be an arbitrary non-empty set. We can consider the sup-lattice $\la Q^X, \bigvee \ra$, with pointwise defined join, and define a scalar multiplication in $Q^X$ as follows:
$$\cdot : (a,f) \in Q \times Q^X \lmapsto a \cdot f \in Q^X,$$
with the map $a \cdot f$ defined as $(a \cdot f)(x) = a \cdot f(x)$ for all $x \in X$. Then $Q^X$ is a left $\QQ$-module~---~denoted by $\QQ^X$~---~and, for all $a \in Q$, $f \in Q^X$ and $x \in X$, $(a \ust f)(x) = a \under f(x)$. It is well-known that, if $1 \neq \bot$ in $\QQ$, $\QQ^X$ is the free $\QQ$-module over $X$. In what follows, for any subset $S$ of $Q^X$, we shall denote by $\QQ \cdot S$ the submodule of $\QQ^X$ generated by $S$.

Let $\AA = \la A, \cdot, 1 \ra$ be a monoid. Then $\wp(\AA) = \la \wp(A), \bigcup, \cdot, \{1\} \ra$ is a quantale, with the multiplication defined by $B C = \{bc \mid b \in B, c \in C\}$, for all $B, C \subseteq A$. It is in fact the free quantale over the monoid $\AA$. Indeed, looking at $\Q$ as a concrete category over the one of monoids $\cat M$, such a construction defines a functor which is easily seen to be left adjoint to the forgetful functor from $\Q$ to $\cat M$.

Now let $S$ be an $\AA$-set. Then the sup-lattice $\wp(\mathbf{S}) = \la \wp(S), \bigcup\ra$ is a $\wp(\AA)$-module with the scalar multiplication defined by $B \cdot X = \{b \cdot x \mid b \in B, x \in X\}$, for all $B \subseteq A$ and for all $X \subseteq S$. Indeed it is easy to see that such an operation preserves arbitrary unions in both arguments.

The definition and properties of right $\QQ$-modules are completely analogous. If $\QQ$ is commutative, the concepts of right and left $\QQ$-modules coincide and we will say simply $\QQ$-modules. If a sup-lattice $\MM$ is both a left $\QQ$-module and a right $\RR$-module~---~over two given quantales $\QQ$ and $\RR$~---~we will say that $\MM$ is a \emph{$(\QQ,\RR)$-bimodule} if the following associative law holds:
\begin{equation}\label{bimodule}
(a \cdot_Q x) \cdot_R b = a \cdot_Q (x \cdot_R b), \quad \textrm{for all } \ x \in M, \ a \in Q, \ b \in R.
\end{equation}

\begin{definition}\label{nucleus}
A \emph{(left) $\QQ$-module nucleus} (or \emph{structural closure operator}) $\g$ over $\MM$ is a closure operator such that $a \cdot \g(x) \leq \g(a \cdot x)$, for all $a \in Q$ and $x \in M$. If $\g$ is a nucleus, we will denote by $M_\g$ the $\g$-closed system $\g[M]$, and it is easy to see that $M_\g$ is closed under arbitrary meets; moreover, $\MM_\g$ is a left $\QQ$-module itself, as shown in Theorem \ref{homoandnuclei}.
\end{definition}

\begin{remark*}
Henceforth, in all the definitions and results that can be stated both for left and right modules, we will refer generically to ``modules''~--- without specifying left or right~--- and we will use the notations of left modules.
\end{remark*}

If $\MM$ and $\NN$ are modules over the same quantale $\QQ$, a map $f: M \lto N$ is a \emph{$\QQ$-module homomorphism} if it is a sup-lattice homomorphism, i.e., a residuated map, and preserves the scalar multiplication: $f(a \cdot_M x) = a \cdot_N f(x)$ for all $a \in Q$ and $x \in M$.

\begin{theorem}\cite{galtsi}\label{homoandnuclei}
Let $\QQ$ be a quantale, $\MM$ and $\NN$ $\QQ$-modules, and $f \in \hom_\QQ(\MM,\NN)$. Then $f_* \circ f$ is a nucleus on $\MM$.

Conversely, if $\g$ is a nucleus on $\MM$, then $M_\g$~--- with the join $\bigvee_\g = \g \circ \bigvee$ and the external product $\cdot_\g = \g \circ \cdot$~--- is a $\QQ$-module (denoted by $\MM_\g$), and there exists $f_\g \in \hom_\QQ(\MM, \MM_\g)$ such that ${f_\g}_* \circ f_\g = \g$.
\end{theorem}

For further insights on algebraic and categorical properties of quantale modules the reader may refer to \cite{pase,russo,thesis,solo}.

\section{Interpretations between systems with the same language}
\label{inter}

As we pointed out in the previous section, if $\AA$ is a monoid and $S$ is an $\AA$-set, $\wp(\AA)$ is a quantale and $\wp(\SS)$ is a $\wp(\AA)$-module. As a consequence of this simple observation, Lemma \ref{oprel}, and condition (\ref{subinv}), we obtain the following immediate result.
\begin{proposition}\label{oprelnuc}
Let $\AA$ be a monoid and $S$ an $\AA$-set. If $\vdash$ is an action-invariant consequence relation on $S$, then the map \begin{equation}\label{subvdash}
\g_\vdash: X \in \wp(S) \lmapsto \{u \in S \mid X \vdash u\} \in \wp(S)
\end{equation}
is a $\wp(\AA)$-module nucleus on $\wp(\SS)$. Conversely, if $\gamma$ is a $\wp(\AA)$-module nucleus on $\wp(\SS)$, the relation $\vdash_\g \subseteq \wp(S) \times S$ defined by
\begin{equation}\label{subgamma}
X \vdash_\g u \textrm{ iff } u \in \g(X)
\end{equation}
is an action-invariant consequence relation on $S$.
\end{proposition}

Specializing Proposition \ref{oprelnuc} to the concrete case of a substitution invariant deductive system $\la S, \vdash\ra$ over a set of sequents on a given language $\lang$, we have that $\wp(\SS)$ is a $\wp(\Sfm)$-module and the consequence relations on $S$ are in one-one correspondence with the $\wp(\Sfm)$-module nuclei on $\wp(\SS)$ or, equivalently by Theorem \ref{homoandnuclei}, with the quotients of such a module.

In \cite{galtsi}, Galatos and Tsinakis used this approach\footnote{Actually they did not use exactly the notation and terminology of quantales and quantale modules but the categories they introduced are precisely the same.} to show that, given two substitution invariant deductive systems $\la S, \vdash_S \ra$ and $\la T, \vdash_T\ra$ over a given language $\lang$, the fact that the $\wp(\Sfm)$-modules $\wp(\SS)_{\g_{\vdash_S}}$ and $\wp(\TT)_{\g_{\vdash_T}}$ are isomorphic is a necessary and sufficient condition for the existence of an equivalence between the two corresponding deductive systems. This result is a consequence of the following Propositions \ref{gt1}, \ref{gt2}, and \ref{gt3}, and we will extend it to interpretations and representations in Theorem \ref{gt}.
\begin{proposition}\label{gt1}
A $\QQ$-module $\MM$ is cyclic and projective if and only if it is isomorphic to the module $\QQ \cdot u$ for some multiplicatively idempotent element $u$ of $Q$.
\end{proposition}
\begin{proof}
The assertion is simply a reformulation of the equivalence between conditions 4 and 5 of \cite[Theorem 5.7]{galtsi}.
\end{proof}

\begin{proposition}\label{gt2}
Let $\{\MM_i\}_{i \in I}$ be a family of $\QQ$-modules. The (co)product $\MM$ of the family $\{\MM_i\}_{i \in I}$ is projective if and only if $\MM_i$ is projective for all $i \in I$.
\end{proposition}
\begin{proof}
One implication is a consequence of the fact that product and coproduct of the same family of modules have the same underlying object, namely, the direct product with pointwise defined operations, and therefore each $\MM_i$ is a retract of $\MM$. The converse implication is proved in \cite[Lemma 5.12]{galtsi}.
\end{proof}

\begin{proposition}\label{gt3}
For any propositional language $\lang$, if $S$ is a set of sequents of a single type, then $\wp(\SS)$ is a cyclic projective $\wp(\Sfm)$-module, if $S$ has more than one type, then $\wp(\SS)$ is the coproduct of cyclic projective modules. In particular $\wp(\SS)$ is always a projective module.
\end{proposition}
\begin{proof}
See Corollary 5.9 and Theorem 5.13 of \cite{galtsi}.
\end{proof}

In particular Proposition \ref{gt3} asserts that $\wp(\Fml)$ and $\wp(\EQ)$ are cyclic projective modules over $\wp(\Sfm)$.

The part of next result concerning equivalences is a direct consequence of the results of \cite[Sections 5 and 6]{galtsi}.
\begin{theorem}\label{gt}
Let $\lang$ be a propositional language, $\cat S =  \la S, \vdash_S\ra$ and $\cat T = \la T, \vdash_T\ra$ deductive systems over the sets of sequents $S$ and $T$ respectively, and $\gamma = \gamma_{\vdash_S}$ and $\d = \gamma_{\vdash_T}$. Then $\cat S$ is interpretable in (respectively: representable in, equivalent to) $\cat T$ if and only if there exists a $\wp(\Sfm)$-module homomorphism (resp.: injective homomorphism, isomorphism) $f: \wp(\SS)_\g \lto \wp(\TT)_\d$.
\end{theorem}
\begin{proof}
\textrm{}
\begin{itemize}
\item[$(\Longrightarrow)$]
\end{itemize}
Let $\iota: S \lto \wp(T)$ be an interpretation. We define
$$f: \wp(\SS)_\g \lto \wp(\TT)_\d$$
by
\begin{equation}\label{f}
f(\g(\Phi)) = \d(\iota[\Phi]), \quad \textrm{for all $\Phi \in \wp(S)$}.
\end{equation}
If $\Phi, \Psi \in \wp(S)$ are such that $\g(\Phi) = \g(\Psi)$, then $\Phi \vdash_S \Xi$ for all $\Xi \subseteq \g(\Psi)$, hence $\iota[\Phi] \vdash_T \iota[\Xi]$ which means that $\d(\iota[\Psi]) \subseteq \d(\iota[\Phi])$. The converse inclusion can be proved analogously, so $\d(\iota[\Psi]) = \d(\iota[\Phi])$ and $f$ is a well-defined function.

Let now $\{\Phi_i\}_{i \in I} \subseteq \wp(S)$. We have
$$\begin{array}{l}
f\left({}^\g\bigvee_{i \in I}\g(\Phi_i)\right) = f\left(\g\left(\bigcup_{i \in I} \Phi_i\right)\right) \\
= \d\left(\iota\left[\bigcup_{i \in I} \Phi_i\right]\right) = \d\left(\bigcup_{i \in I}\iota[\Phi_i]\right) \\
= {}^{\d}\bigvee_{i \in I} \d\left(\iota[\Phi_i]\right) = {}^{\d}\bigvee_{i \in I} f(\g(\Phi_i)).
\end{array}$$
Then $f$ is a sup-lattice homomorphism.

Last we need to prove that $f$ is a $\wp(\Sfm)$-module homomorphism, so let $\Sigma \in \wp(\sfm)$ and $\Phi \in \wp(S)$. We have
$$\begin{array}{l}
f(\Sigma \cdot_\g \g(\Phi)) = f(\g(\Sigma \cdot \Phi)) = \d(\iota[\Sigma \cdot \Phi]) \\
= \d(\Sigma \cdot \iota[\Phi]) = \Sigma \cdot_\d \d(\iota[\Phi]) \\
= \Sigma \cdot_\d f(\g(\Phi)).
\end{array}$$

Now, let us assume that $\iota$ is a representation. In order to prove that $f$ is injective, let us consider $\Phi, \Psi \in \wp(S)$ such that $\g(\Phi) \neq \g(\Psi)$; we can assume, without losing generality, that there exists $\phi \in \g(\Phi) \setminus \g(\Psi)$. Then $\Psi \nvdash_S \phi$, hence $\iota[\Psi] \nvdash_T \iota(\phi)$ by (\ref{eqint}). It follows that $f(\g(\Phi)) \neq f(\g(\Psi))$ and therefore $f$ is injective.

The case of equivalence follows immediately from the one of representation and (\ref{eqeq}).

\begin{itemize}
\item[$(\Longleftarrow)$]
\end{itemize}
Conversely, let us assume the existence of $f$ and consider the following diagram of $\wp(\Sfm)$-module morphisms
\begin{equation}\label{concinterdiag1}
\xymatrix{
\wp(\mathbf{S}) \ar@{-->}[rr]^{g} \ar@{->>} [dd]_{\g} & & \wp(\mathbf{T}) \ar@{->>} [dd]^\d\\
&&\\
\wp(\mathbf{S})_\g \ar[rr]_{f} & & \wp(\mathbf{T})_\d,
\\
}
\end{equation}
which can be completed with a morphism $g$ because $\wp(\mathbf{S})$ is a projective module by Proposition \ref{gt3}. Moreover, since powersets are free sup-lattices, $g$ is uniquely determined by its restriction to the singletons, that is, by the map $\iota: \phi \in S \lmapsto g(\{\phi\}) \in \wp(T)$, and such a map is obviously action-invariant.

Now let $\Phi \cup \{\psi\} \subseteq S$. By definition $\Phi \vdash_S \psi$ iff $\{\psi\} \subseteq \g(\Phi)$. So $\g(\{\psi\}) \subseteq \g(\Phi)$ and therefore $f(\g(\{\psi\})) = \d(g(\{\psi\})) \subseteq \d(g(\Phi)) = f(\g(\Phi))$. Hence $\iota[\Phi] \vdash_T \iota(\psi)$, i.e. $\iota$ is an interpretation. Moreover, if $f$ is injective we have: $\g(\{\psi\}) \subseteq \g(\Phi)$ iff $f(\g(\{\psi\})) = \d(g(\{\psi\})) \subseteq \d(g(\Phi)) = f(\g(\Phi))$ iff $\iota[\Phi] \vdash_T \iota(\psi)$, that is, $\iota$ is a representation. Last, if $f$ is an isomorphism with inverse $f^{-1}$, we can use these two isomorphisms and proceed as above in order to define $\iota$ and $\iota'$ in such a way that Definition \ref{intlang}(iv) is satisfied.

The theorem is proved.
\end{proof}

\section{Restriction and extension of scalars}
\label{tens}

In Section \ref{transsec} we defined translations between propositional languages; we also proved that any translation induces a homomorphism between the substitution monoids of the two languages (Lemma \ref{qinq'}) and, conversely, a monoid homomorphism between monoids of substitutions is induced by a language translation if and only if it satisfies certain conditions (Theorem \ref{transchar}).

As already pointed out, any monoid homomorphism $f: \AA \lto \BB$ extends to a unique quantale homomorphism $h_f: \wp(\AA) \lto \wp(\BB)$; therefore a language translation from $\lang$ to $\lang'$, in the algebraic representation introduced in the previous section, can be viewed as a quantale morphism between $\wp(\Sfm)$ and $\wp(\Sfmm)$.

It is well-known in ring theory that a homomorphism between two rings canonically induces a functor (in the opposite direction) between the corresponding categories of modules. In this section we shall extend such a construction and its properties to quantales. Most of the results on tensor products we are going to present in this section already appeared in \cite{thesis}. Moreover, some of them are analogous to the corresponding classical results (not only in ring theory) \cite{ban,lin,joyaltierney} or can be deduced as special cases of more general ones obtained either in the general categorical setting \cite{kat1,kat2} or --- more recently --- in the context of quantaloids \cite{galgil}. Nonetheless, all the proofs presented are necessary for the sake of readability. Indeed, the proof of Theorem \ref{subs} and the applications presented in Section \ref{interabs} need some technical details which appear here and there in those proofs.

\begin{lemma}\label{indmod}
Let $\QQ$ and $\RR$ be quantales and $h: \QQ \lto \RR$ a quantale homomorphism. Then $h$ induces a structure of $\QQ$-module on each $\RR$-module; in particular, $h$ induces structures of $\QQ$-bimodule, $\RR$-$\QQ$-bimodule and $\QQ$-$\RR$-bimodule on $\RR$ itself. Moreover, w.r.t. such a structure of $\QQ$-module on $\RR$ and the one of free cyclic $\QQ$-module on $\QQ$, $h$ is also a $\QQ$-module homomorphism.
\end{lemma}
\begin{proof}
Let $\NN = \la N, \bigvee \ra$ be an $\RR$-module. It is easy to verify that
\begin{equation}\label{starh}
\cdot_h: (a, x) \in Q \times N \lmapsto h(a) \cdot x \in N
\end{equation}
makes $\NN$ into a $\QQ$-module, henceforth denoted by $\NN_h$. Since $\RR$ is a bimodule over itself, the second part of the assertion follows immediately.

Last, for any $a, b \in Q$, $h(ab) = h(a)h(b) = a \cdot_h h(b)$, hence $h$ is a $\QQ$-module homomorphism between $\QQ$ and $\RR_h$. 
\end{proof}

The operation performed in (\ref{starh}) is well-known in the theory of ring modules as \emph{restricting the scalars along $h$} (see, e.g., \cite{ban}). In fact it defines a faithful functor
\begin{equation}\label{subh}
\begin{array}{cccc}
\sub_h: & \RR\Mod & \lto & \QQ\Mod \\
& \NN & \lmapsto & \NN_h
\end{array}
\end{equation}
having both a right and a left adjoint. This property was proved to hold for commutative quantales by Joyal and Tierney \cite{joyaltierney}. In the general case, however, the situation is precisely the same, as we are going to show.

\begin{definition}\label{tensormq}
Let $\QQ$ be a quantale, $\MM_1 = \la M_1, \bigvee_1\ra$ a right $\QQ$-module, $\MM_2 = \la M_2, \bigvee_2\ra$ a left $\QQ$-module, and $\LL = \la L, \bigvee\ra$ a sup-lattice. Then $\MM_1 \times \MM_2$ is a $\QQ$-bimodule, where the join is defined componentwise, and left and right scalar multiplications are defined, for all $(x,y) \in M_1 \times M_2$ and $a \in Q$, respectively as follows:
\begin{enumerate}
\item[]$a \cdot_l (x,y) = (x,a \cdot_2 y)$,
\item[]$(x,y) \cdot_r a = (x \cdot_1 a,y)$.
\end{enumerate}

A map $f: M_1 \times M_2 \lto L$ is said to be a \emph{$\QQ$-bimorphism} if it preserves arbitrary joins in each variable separately
\begin{eqnarray*}
&& f\left(\bigvee X, y\right) = \bigvee_{x \in X} f(x,y), \\
&& f\left(x,\bigvee Y\right) = \bigvee_{y \in Y} f(x,y), \\
\end{eqnarray*}
and
\begin{equation}\label{bimproduct}
f(x, a \cdot_2 y) = f(x \cdot_1 a,y).
\end{equation}

The \emph{tensor product} $\MM_1 \tensor_\QQ \MM_2$, of the $\QQ$-modules $\MM_1$ and $\MM_2$, is the codomain of the universal $\QQ$-bimorphism $\MM_1 \times \MM_2 \lto \MM_1 \tensor_\QQ \MM_2$. In other words, we call tensor product of $\MM_1$ and $\MM_2$ a sup-lattice $\MM_1 \tensor_\QQ \MM_2$ equipped with a $\QQ$-bimorphism $\tau: \MM_1 \times \MM_2 \lto \MM_1 \tensor_\QQ \MM_2$ such that, for any sup-lattice $\LL$ and any $\QQ$-bimorphism $f: \MM_1 \times \MM_2 \lto \LL$, there exists a unique sup-lattice homomorphism $k_f: \MM_1 \tensor_\QQ \MM_2 \lto \LL$ satisfying $k_f \circ \tau = f$.
\end{definition}

\begin{theorem}\label{tensormqexists}
Let $\MM_1$ be a right $\QQ$-module and $\MM_2$ a left $\QQ$-module. Then the tensor product $\MM_1 \tensor_\QQ \MM_2$ of the $\QQ$-modules $\MM_1$ and $\MM_2$ exists. It is, up to isomorphisms, the quotient $\wp(\MM_1 \times \MM_2)/\theta_R$ of the free sup-lattice generated by $M_1 \times M_2$ with respect to the (sup-lattice) congruence relation generated by the set
\begin{equation}\label{R}
R = \left\{
	\begin{array}{l}
		\left(\left\{\left(\bigvee X, y\right)\right\}, \bigcup_{x \in X}\{(x,y)\}\right) \\
		\left(\left\{\left(x, \bigvee Y\right)\right\}, \bigcup_{y \in Y}\{(x,y)\}\right) \\
		\left(\{(x \cdot_1 a, y)\}, \{(x,a \cdot_2 y)\}\right) \\
	\end{array} \right\vert
	\left.
	\begin{array}{l}
	X \subseteq M_1, y \in M_2 \\
	Y \subseteq M_2, x \in M_1 \\
	a \in Q \\
	\end{array}
	\right\}.
\end{equation}
\end{theorem}
\begin{proof}
Let $\LL$ be any sup-lattice and $f: \MM_1 \times \MM_2 \lto \LL$ be a $\QQ$-bimorphism. Then we can extend the map $f$ to a sup-lattice homomorphism $h_f: \wp(\MM_1 \times \MM_2) \lto \LL$; thus $h_f \circ \sigma = f$, where $\sigma: M_1 \times M_2 \lto \wp(M_1 \times M_2)$ is the singleton map. On the other hand, the fact that $f$ is a $\QQ$-bimorphism implies $f\left(\bigvee X, v\right) = \bigvee_{x \in X} f(x,v)$, $f\left(u,\bigvee Y\right) = \bigvee_{y \in Y} f(u,y)$, and $f(u \cdot_1 a, v) = f(u, a \cdot_2 v)$, for all $X \cup \{u\} \subseteq M_1$, $Y \cup \{v\} \subseteq M_2$, and $a \in Q$. Now, since $h_f$ is a sup-lattice homomorphism, we have $h_f\left(\left\{\left(\bigvee X,v\right)\right\}\right) = h_f\left(\bigcup_{x \in X}\{(x,v)\}\right)$ and $h_f\left(\left\{\left(u,\bigvee Y\right)\right\}\right) = h_f\left(\bigcup_{y \in Y}\{(u,y)\}\right)$. Moreover, we have
$$\begin{array}{rcl}
&&h_f(\{(u \cdot_1 a, v)\}) = (h_f \circ \sigma)(u \cdot_1 a, v) = f(u \cdot_1 a, v) \\
&=& f(u,a \cdot_2 v) = (h_f \circ \sigma)(u,a \cdot_2 v) = h_f(\{(u,a \cdot_2 v)\}). \\ 
\end{array}$$

Hence the kernel of $h_f$ contains $R$ and~--- once denoted by $\mathbf{T}$ the quotient sup-lattice $\wp(\MM_1 \times \MM_2)/\theta_{R}$ and by $\pi$ the canonical quotient morphism of $\wp(\MM_1 \times \MM_2)$ over  it~--- the map
$$k_f \ : \quad X/\theta_R \ \in \ \mathbf{T} \quad \lmapsto \quad h_f(X) \ \in \ \LL$$
is a well-defined sup-lattice homomorphism and $k_f \circ \pi \circ \sigma = h_f \circ \sigma = f$. So we extended the $\QQ$-bimorphism $f$ to a sup-lattice homomorphism $k_f$, and it is immediate to verify both that $\tau = \pi \circ \sigma$ is a bimorphism and that such a $k_f$ is necessarily unique.

The following commutative diagram should better illustrate the construction.
\begin{equation}\label{tensormqdiagram}
\xymatrix{
\MM_1 \times \MM_2 \ar[rr]^\sigma \ar[rddd]_f \ar[rd]_{\tau} && \wp(\MM_1 \times \MM_2) \ar[lddd]^{h_f} \ar[ld]^{\pi} \\
 & \mathbf{T}  \ar[dd]^{k_f} & \\
 &&\\
 & \LL & \\
}
\end{equation}

It is important to remark explicitly that $R$ and $\tau$ depend neither on the sup-lattice $\LL$ nor on the $\QQ$-bimorphism $f$. So we proved that $\tau$ is the universal bimorphism whose domain is $\MM_1 \times \MM_2$, and that $\mathbf{T}$ is its codomain, i.e., the tensor product $\MM_1 \tensor_\QQ \MM_2$ of the $\QQ$-modules $\MM_1$ and $\MM_2$.
\end{proof}

For all $x \in M_1$ and $y \in M_2$, we will denote by $x \tensor y$ the image of the pair $(x,y)$ under $\tau$, i.e., the congruence class $\{(x,y)\}/\theta_R$, and we will call it a \emph{$\QQ$-tensor} or, if there will not be danger of confusion, simply a \emph{tensor}. It is clear then that every element of $\MM_1 \tensor_\QQ \MM_2$ is a join of tensors, so
$$\MM_1 \tensor_\QQ \MM_2 = \left\{\bigvee_{i \in I} x_i \tensor y_i \ \Big\vert \ x_i \in M_1, y_i \in M_2\right\}.$$

Let now $\QQ$ and $\RR$ be two quantales, if $\MM_1$ is an $\RR$-$\QQ$-bimodule and $\MM_2$ is a left $\QQ$-module, then the tensor product $\MM_1 \tensor_\QQ \MM_2$ naturally inherits a structure of left $\RR$-module from the one defined on $\MM_1$:
\begin{equation*}
\star_l: \ \left(b, \bigvee_{i\in I} x_i \tensor y_i\right) \in \RR \times \left(\MM_1 \tensor \MM_2\right) \ \lmapsto \ \bigvee_{i \in I} (b \cdot_\RR x_i) \tensor y_i \in \MM_1 \tensor \MM_2.
\end{equation*}
Indeed it is trivial that $\star_l$ distributes over arbitrary joins in both coordinates; on the other hand, the external associative law comes straightforwardly from the fact that $\MM_1$ is a left $\RR$-module. Analogously, if $\MM_1$ is a right $\QQ$-module and $\MM_2$ is a $\QQ$-$\RR$-bimodule, then the tensor product $\MM_1 \tensor_\QQ \MM_2$ is a right $\RR$-module with the scalar multiplication defined, obviously, as
\begin{equation*}
\star_r: \ \left(\bigvee_{i \in I} x_i \tensor y_i, b\right) \in \left(\MM_1 \tensor \MM_2\right) \times \RR \ \lmapsto \ \bigvee_{i \in I} x_i \tensor (y_i \cdot_\RR b) \in \MM_1 \tensor \MM_2.
\end{equation*}

The following Lemmas \ref{diam}, \ref{isohomtensmq'}, and \ref{homqm} are not directly related to this work, but they are involved in the proof of Theorem \ref{adjfunct}. So we report them here and refer to the corresponding results of \cite{thesis} for the proofs of the first two, which can be skipped without compromising the comprehension of the subsequent material.

\begin{lemma}\label{diam}\cite[Theorem 4.7.4]{thesis}
Let $\QQ$ and $\RR$ be quantales. If $\MM_1$ is a $\QQ$-$\RR$-bimodule and $\MM_2$ is a left $\QQ$-module, then $\Hom_\QQ(\MM_1,\MM_2)$ is a left $\RR$-module with the external product $\bullet_l$ defined, for $b \in R$, $h \in \hom_\QQ(\MM_1,\MM_2)$ and $x \in M_1$, by
\begin{equation}\label{homleftq}
(b \bullet_l h)(x) = h(x \cdot_\RR b),
\end{equation}
$\cdot_\RR$ denoting the right external product of $\MM_1$.

Analogously, if $\MM_1$ is an $\RR$-$\QQ$-bimodule and $\MM_2$ is a right $\QQ$-module, then $\Hom_\QQ(\MM_1,\MM_2)$ is a right $\RR$-module with the external product $\bullet_r$ defined, for $b \in R$, $h \in \hom_\QQ(\MM_1,\MM_2)$ and $x \in M_1$, by
\begin{equation}\label{homrightq}
(h \bullet_r b)(x) = h(b \cdot_\RR x),
\end{equation}
$\cdot_\RR$ denoting the left external product of $\MM_1$.
\end{lemma}

Analogously we have
\begin{lemma}\label{isohomtensmq'}\cite[Theorem 4.7.6]{thesis}
Let $\QQ$ and $\RR$ be quantales and let $\MM_1$ be a $\RR$-$\QQ$-bimodule, $\MM_2$ a left $\QQ$-module and $\MM_3$ a left $\RR$-module. Then, if we consider the left $\RR$-module $\MM_1 \tensor_\QQ \MM_2$ and the left $\QQ$-module $\Hom_\RR(\MM_1,\MM_3)$, we have
$$\Hom_\RR(\MM_1 \tensor_\QQ \MM_2, \MM_3) \cong_{\SL} \Hom_\QQ(\MM_2,\Hom_\RR(\MM_1,\MM_3)),$$
where $\cong_{\SL}$ means that they are isomorphic as sup-lattices.
\end{lemma}

\begin{lemma}\label{homqm}
Let $\QQ$ be a quantale and $\MM$ be a $\QQ$-module. Then, considering $\QQ = \la Q, \bigvee \ra$ as a module over itself, we have
\begin{equation*}
\Hom_\QQ(\QQ,\MM) \cong_{\SL} \MM.
\end{equation*}
\end{lemma}
\begin{proof}
First of all we observe that, for any fixed $x \in M$, the map $f_x: a \in Q \lto a \cdot x \in M$ is trivially a $\QQ$-module homomorphism. Then we can consider the map $\alpha: x \in M \lto f_x \in \hom_\QQ(\QQ,\MM)$, which is clearly a sup-lattice homomorphism.

Let us consider also the map $\beta: f \in \hom_\QQ(\QQ,\MM) \lto f(1) \in M$. Again, it is immediate to verify that $\beta$ is a sup-lattice homomorphism. But we also have:
$$((\alpha \circ \beta)(f))(a) = (\alpha(f(1)))(a) = f_{f(1)}(a) = a \cdot f(1) = f(a),$$
for all $f \in \hom_\QQ(\QQ,\MM)$ and $a \in Q$, and
$$(\beta \circ \alpha)(x) = \beta(f_x) = f_x(1) = 1 \cdot x = x,$$
for all $x \in M$.

Thus $\alpha \circ \beta = \id_{\hom_\QQ(\QQ,\MM)}$ and $\beta \circ \alpha = \id_M$, i.e. $\alpha$ is a sup-lattice isomorphism whose inverse is $\beta$, and the thesis follows.
\end{proof}

As a consequence of the previous result, the $\QQ$-module structure defined on $\hom_\QQ(\QQ,\MM)$ by Lemma~\ref{diam} is isomorphic to $\MM$.

Let now $\QQ \leq \RR$ be quantales. If $\MM$ is a left $\QQ$-module, we can use the tensor product in order to extend the $\QQ$-module $\MM$ to an $\RR$-module. Indeed, if we consider $\RR$ as an $\RR$-$\QQ$-bimodule, the $\QQ$-tensor product $\RR \tensor_\QQ \MM$ is a left $\RR$-module (hence, also a left $\QQ$-module).

Let $x \in M$; for any $a \in Q$,
$$a \star_l (1 \tensor x) = a \tensor x = (1 \cdot a) \tensor x = 1 \tensor (a \cdot_M x).$$
So the set $1 \tensor_\QQ \MM = \{1 \tensor x \mid x \in M\}$~--- that clearly generates $\RR \tensor_\QQ \MM$ as $\RR$-module~--- is a $\QQ$-submodule of $\RR \tensor_\QQ \MM$, and a homomorphic image of $\MM$. The map
\begin{equation*}
\begin{array}{llll}
1 \tensor \iota_\MM : \quad & \MM \quad & \lto \quad 	& \RR \tensor_\QQ \MM \\
										 				&	x  				& \lmapsto		& 1 \tensor x \\
\end{array}
\end{equation*}
is a $\QQ$-module homomorphism.

If $\MM = \QQ^X$ is a free module, the tensor product is isomorphic to the free $\RR$-module over the same generating set: $\RR \tensor_\QQ \QQ^X \cong \RR^X$. Indeed the map $\phi: (b,f) \in \RR \times \QQ^X \lmapsto (b \cdot f(x))_{x \in X} \in \RR^X$ is clearly a $\QQ$-bimorphism, and the homomorphism that extends $\phi$ to $\RR \tensor_\QQ \QQ^X$ is
$$k_\phi: \bigvee_{i \in I} b_i \tensor f_i \in \RR \tensor_\QQ \QQ^X \lmapsto \bigvee_{i \in I} b_i \cdot f_i \in \RR^X.$$
Then, setting $\chi_x(y) = \left\{\begin{array}{ll} \bot & \textrm{if } y \neq x \\ 1 & \textrm{if } y = x \end{array}\right.$ for all $x \in X$, it is easy to verify that $k': g \in \RR^X \lmapsto \bigvee_{x \in X} g(x) \tensor \chi_x \in \RR \tensor_\QQ \QQ^X$ is a homomorphism and it is the inverse of $k_\phi$. Every element of $\RR \tensor_\QQ \QQ^X$ can be written in a unique way as $\bigvee_{x \in X} b_x (1 \tensor \chi_x)$, i.e. $\RR \tensor_\QQ \QQ^X$ is the free $\RR$-module generated by the set $\{1 \tensor \chi_x \mid x \in X\}$, equipotent to $X$.

In general, if $\MM$ is a left $\QQ$-module, $X$ is a set of generators for $\MM$, and $\RR$ is a quantale containing $\QQ$, then the left $\RR$-module $\RR \tensor_\QQ \MM$ is generated by $1 \tensor_\QQ X = \{1 \tensor x \mid x \in X\}$.

\begin{theorem}\label{adjfunct}
The functor $\sub_h$ defined in (\ref{subh}) has both a left adjoint $\subl_h$ and a right adjoint $\subr_h$.
\end{theorem}
\begin{proof}
For any $\MM \in \QQ\Mod$, viewing $\RR$ as an $\RR$-$\QQ$-bimodule, we can construct the tensor product $\RR \tensor_{\QQ} \MM$ which is a left $\RR$-module. We claim that the left adjoint of $\sub_h$ is
\begin{equation}\label{hl}
\begin{array}{cccc}
\subl_h: & \QQ\Mod & \lto & \RR\Mod \\
			& \MM			& \lmapsto & \RR \tensor_{\QQ} \MM
\end{array}.
\end{equation}
In order to prove that, we need to show that there exists a natural bijection between $\hom_\RR(\RR \tensor_{\QQ} \MM, \NN)$ and $\hom_\QQ(\MM,\NN_h)$, for all $\MM \in \QQ\Mod$ and $\NN \in \RR\Mod$. The hom-set $\hom_\RR(\RR \tensor_{\QQ} \MM, \NN)$ is isomorphic, as a sup lattice, to $\hom_\QQ(\MM, \hom_\RR(\RR,\NN))$, by Theorem~\ref{isohomtensmq'}; on the other hand, $\hom_\RR(\RR,\NN)$ and $\NN$ are isomorphic sup-lattices by Lemma~\ref{homqm}, and such an isomorphism is a $\QQ$-module isomorphism (with $\NN_h$ instead of $\NN$) for how the $\QQ$-module structure is induced on $\hom_\RR(\RR,\NN)$. Hence the two hom-sets are isomorphic sup-lattices, and $\subl_h$ is the left adjoint of $\sub_h$. 

The right adjoint is defined by
\begin{equation}\label{hr}
\begin{array}{cccc}
\subr_h: & \QQ\Mod & \lto & \RR\Mod \\
			 & \MM & \lmapsto & \Hom_\QQ(\RR_h,\MM),
\end{array}
\end{equation}
where the left $\RR$-module structure on $\Hom_\QQ(\RR_h,\MM)$ is the one introduced in Lemma~\ref{diam}. This part of the proof is analogous to the case of $\subl_h$. Indeed, for any $\QQ$-module $\MM$ and any $\RR$-module $\NN$, by Theorem~\ref{isohomtensmq'}, $\hom_\RR(\NN,\MM_h^r)$ --- namely $\hom_\RR(\NN, \Hom_\QQ(\RR_h,\MM))$ --- is isomorphic to $\hom_\QQ((\RR \tensor_\RR \NN)_h, \MM)$ in $\SL$; on the other hand, since every tensor $b \tensor y \in \RR \tensor_\RR \NN$ can be rewritten in the form $1 \tensor b \cdot_N y$, such a tensor product is easily seen to be isomorphic to $\NN_h$. Therefore $\hom_\RR(\NN,\MM_h^r)$ is a sup-lattice isomorphic to $\hom_\QQ(\NN_h, \MM)$ and the theorem is proved.
\end{proof}

\begin{theorem}\label{subs}
Let $h: \QQ \lto \RR$ be an onto quantale homomorphism. Then the functor $\sub_h$ is a full embedding and its left adjoint $\RR_h \tensor_\QQ {}\<$ is, up to a natural isomorphism, its left inverse, i.e., $\subl_h \circ \sub_h$ is naturally isomorphic to $\ID_{\RR\Mod}$.

Moreover, if $h$ is a retraction with correspoding section $k$, $\RR_h \tensor_\QQ {}\<$ and $\sub_k$ are naturally isomorphic.
\end{theorem}
\begin{proof}
We already observed that the functor $\sub_h$ is faithful for any quantale morphism $h$. If $h$ is onto, $\sub_h$ is obviously injective on objects. Indeed, on the one hand, it does not affect the underlying sup-lattice structure; on the other hand, if $\MM$ and $\MM'$ are two different $\RR$-module structures with the same underlying sup-lattice $M$, then there exist $b \in R$ and $x \in M$ such that $b \cdot x \neq b \cdot' x$. Hence, for any $a \in h^{-1}(b) \neq \varnothing$, $a \cdot_h x \neq a \cdot_h' x$ and therefore $\MM_h \not\cong_{\QQ\Mod} \MM_h'$. So $\sub_h$ is a categorical embedding.

Moreover, for any $\QQ$-module homomorphism $f: \MM_h \lto \NN_h$, the same underlying function is also an $\RR$-module homomorphism between $\MM$ and $\NN$, hence $f$ is the image under $\sub_h$ of an $\RR$-module morphism, and this means that $\sub_h$ is full.

Now let $\NN \in \RR\Mod$ and consider the $\QQ$-module morphism $1 \tensor {}\<: y \in \NN_h \lmapsto 1 \tensor y \in \RR_h \tensor_\QQ \NN_h$. It is easy to see that under the given hypotheses $1 \tensor {}\<$ is a $\QQ$-module isomorphism, hence $\subl_h \circ \sub_h$ is naturally isomorphic to the identity functor $\ID_{\RR\Mod}$. In particular, if $\RR$ is a retract of $\QQ$ under $h$ and $k$, $\sub_{h \circ k} = \sub_k \circ \sub_h = \ID_{\RR\Mod}$. So $\sub_k$ is the left inverse of $\sub_h$ and therefore it is naturally isomorphic to $\RR_h \tensor_\QQ {}\<$.
\end{proof}

Theorem \ref{transchar} provides a characterization of homomorphisms between substitution monoids which are induced by a language translation. If we look at the quantales of type $\wp(\Sfm)$, a characterization of quantale homomorphisms induced by language translations immediately follows from Theorem \ref{transchar}. Indeed we just need to observe the following two facts:
\begin{itemize}
\item any quantale of substitutions contains a subquantale isomorphic to $\wp(\VV)$, and
\item for any monoid $\AA$, the completely join-prime elements of the quantale $\wp(\AA)$ are precisely the singletons of elements of $A$.
\end{itemize}

\begin{corollary}\label{transchar2}
Let $h: \wp(\Sfm) \lto \wp(\Sfmm)$ be a quantale homomorphism. Then $h$ is induced by a language translation of $\lang$ in $\lang'$ if and only if it satisfies the following conditions:
\begin{enumerate}[(i)]
\item $h$ preserves the property of being completely join-prime;
\item if $\Sigma \in \wp(\sfmm)$ is completely join-prime and multipicatively idempotent, then $h^{-1}(\Sigma)$ is either empty or is comprised of completely join-prime idempotent elements of $\wp(\Sfm)$;
\item $h^{-1}(\Sigma) = \{\Sigma\}$ for all $\Sigma \in \wp(\VV)$.
\end{enumerate}
\end{corollary}

In what follows, we shall call \emph{(quantale) translations} all the homomorphisms, between any pair of quantales, satisfying the three conditions of Corollary \ref{transchar2}. It is immediate to verify that the composition of two quantale translations is still a quantale translation.

\section{Interpretations between systems with different languages}
\label{interabs}

In the present section we apply the results of the previous one in order to characterize the various types of interpretations between propositional deductive systems over different languages.

\begin{theorem}\label{faiththm}
Let $\cat S = \la S, \vdash_\g\ra$ and $\cat T = \la T, \vdash_\d\ra$ be two propositional deductive systems on $\lang$ and $\lang'$ respectively. Then $\cat S$ is interpretable (respectively: representable) in $\cat T$ if and only if there exist a quantale translation $h: \wp(\Sfm) \lto \wp(\Sfmm)$ and a $\wp(\Sfm)$-module morphism (resp.: an injective $\wp(\Sfm)$-module morphism) $f: \wp(\mathbf{S})_\g \lto (\wp(\mathbf{T})_{\d})_{h}$.
\end{theorem}
\begin{proof}
Assume that $\cat S$ is interpretable in $\cat T$. By Definition \ref{ints}, there exist a translation $\tau: \lang \lto \lang'$ (which induces a monoid homomorphism $\ov\tau: \Sfm \lto \Sfmm$) and a $\ov\tau$-action-invariant map $\iota$ such that (\ref{neqint}) holds. Let $h: \wp(\Sfm) \lto \wp(\Sfmm)$ be the quantale translation determined by $\ov\tau$ and let
$$f: \wp(S)_\gamma \lto \wp(T)_{\d}$$
be defined as follows
\begin{equation}\label{hf}
f(\g(\Phi)) = \d(\iota[\Phi]), \quad \textrm{for all $\Phi \in \wp(S)$}.
\end{equation}
The fact that $f$ is a well-defined sup-lattice homomorphism can be proved exactly as in Theorem \ref{gt}. In order to prove that it is also a $\wp(\Sfm)$-module homomorphism between $\wp(\mathbf{S})_\g$ and $(\wp(\mathbf{T})_{\d})_{h}$, let $\Sigma \in \wp(\sfm)$ and $\Phi \in \wp(S)$. We have
$$\begin{array}{l}
f(\Sigma \cdot_\g \g(\Phi)) = f(\g(\Sigma \cdot \Phi)) = \d(\iota[\Sigma \cdot \Phi]) \\
= \d(\ov\tau[\Sigma] \cdot \iota[\Phi]) = \d(h(\Sigma) \cdot \iota[\Phi])  \\
= h(\Sigma) \cdot_{\d} \d(\iota[\Phi]) = h(\Sigma) \cdot_{\d} f(\g(\Phi)) \\
= \Sigma \ (\cdot_\d)_h \ f(\g(\Phi)).
\end{array}$$
Again as in Theorem \ref{gt}, if $\iota$ is a representation it follows easily from (\ref{eqint}) that $f$ is injective.

Conversely, let us assume the existence of $h$ and $f$. By Corollary~\ref{transchar2} $h$ is the quantale homomorphism extending the monoid homomorphism $\ov\tau: \Sfm \lto \Sfmm$ induced by a language translation $\tau: \fml \lto \fmll$. On the other hand we have the following diagram of $\wp(\Sfm)$-module morphisms
\begin{equation}\label{concinterdiag}
\xymatrix{
\wp(\mathbf{S}) \ar@{-->}[rr]^{g} \ar@{->>} [dd]_{\g} & & \wp(\mathbf{T})_{h} \ar@{->>} [dd]^\d\\
&&\\
\wp(\mathbf{S})_\g \ar[rr]_{f} & & (\wp(\mathbf{T})_{\d})_h,
\\
}
\end{equation}
which can be completed with a morphism $g$ because $\wp(\mathbf{S})$ is a projective module. Moreover, since powersets are also free sup-lattices, $g$ is uniquely determined by its restriction $\iota$ to the singletons. Therefore we have a translation $\tau$ and a $\ov\tau$-action-invariant map $\iota$ satisfying (\ref{neqint}) (resp.: (\ref{eqint}) if $f$ is injective), and the assertion is proved.
\end{proof}

\begin{corollary}\label{equivthm}
Let $\cat S = \la S, \vdash_\g \ra$ and $\cat T = \la T, \vdash_\d\ra$ be two propositional deductive systems on $\lang$ and $\lang'$ respectively. Then $\cat S$ and $\cat T$ are equivalent if and only if there exist two quantale translations $h: \wp(\Sfm) \lto \wp(\Sfmm)$ and $k: \wp(\Sfmm) \lto \wp(\Sfm)$ such that $\wp(\mathbf{S})_\g \cong (\wp(\mathbf{T})_{\d})_h$ in $\wp(\Sfm)\Mod$ and $(\wp(\mathbf{S})_\g)_k \cong \wp(\mathbf{T})_{\d}$ in $\wp(\Sfmm)\Mod$.
\end{corollary}
\begin{proof}
The assertion is an immediate consequence of Theorems \ref{gt} and \ref{faiththm}.
\end{proof}

The next result is an interesting application of Theorems \ref{retraction} and \ref{subs}.
\begin{theorem}\label{equivthm2}
Let $\cat S = \la S, \vdash_\g \ra$ and $\cat T = \la T, \vdash_\d\ra$ be two propositional deductive systems on $\lang$ and $\lang'$ respectively, and $h: \wp(\Sfm) \lto \wp(\Sfmm)$ be a surjective translation.

Then $\cat S$ and $\cat T$ are equivalent (via $h$) if and only if $\wp(\SS)_\g$ and $(\wp(\TT)_\d)_h$ are isomorphic $\wp(\Sfm)$-modules.
\end{theorem}
\begin{proof}
One implication follows trivially from Corollary \ref{equivthm}. Conversely, assume that $f: \wp(\SS)_\g \lto (\wp(\TT)_\d)_h$ is a $\wp(\Sfm)$-module isomorphism with inverse $f^{-1}$, and consider the following diagrams
\begin{equation}\label{equivdiag}
\xymatrix{
\wp(\mathbf{S}) \ar@{-->}[rr]^{g} \ar@{->>} [dd]_{\g} & & \wp(\mathbf{T})_{h} \ar@{->>} [dd]^\d &  \wp(\mathbf{S})_k \ar@{->>} [dd]_{\g} & & \wp(\mathbf{T}) \ar@{-->}[ll]_{g'} \ar@{->>} [dd]^\d\\
&&&&&\\
\wp(\mathbf{S})_\g \ar[rr]_{f} & & (\wp(\mathbf{T})_{\d})_h & (\wp(\mathbf{S})_\g)_k  & & \wp(\mathbf{T})_{\d} \ar[ll]^{f^{-1}}.
\\
}
\end{equation}
The existence of the morphism $g$ which makes the diagram (of $\wp(\Sfm)$-modules) on the left commutative is guaranteed by the previous results. On the other hand, by Theorem \ref{retraction} (applying also Corollary \ref{transchar2}), $h$ is a retraction and therefore there exists a translation $k: \wp(\Sfmm) \lto \wp(\Sfm)$ such that $h \circ k = \id_{\wp(\Sfmm)}$. Since, by Theorem \ref{subs}, $\sub_h$ is a full embedding of $\wp(\Sfmm)\Mod$ into $\wp(\Sfm)\Mod$, $f$ and $f^{-1}$ are also isomorphisms between the $\wp(\Sfmm)$-modules $(\wp(\SS)_\g)_k$ and $\wp(\TT)_\d$. So the diagram on the right hand side of (\ref{equivdiag}) is a diagram of morphisms in $\wp(\Sfmm)\Mod$ and the projectivity of the $\wp(\Sfmm)$-module $\wp(\TT)$ ensures the existence of a morphism $g'$ which makes it commutative. Then the result follows from Corollary \ref{equivthm}.
\end{proof}

We conclude this section with the following characterization of weak interpretations, weak representations and similarities.

\begin{theorem}\label{nonconsthm}
Let $\cat S = \la S, \vdash\ra$ and $\cat T = \la T, \vdash_\d\ra$ be two propositional deductive systems on $\lang$ and $\lang'$ respectively.

Then $\cat S$ is weakly interpretable in (respectively: weakly representable in, similar to) $\cat T$ if and only if there exists a sup-lattice morphism (resp.: injective morphism, isomorphism) $f: \wp(\SS)_\g \lto \wp(\TT)_\d$.
\end{theorem}
\begin{proof}
The proof is basically incorporated in the ones of the previous results. The only difference that is worth mentioning is the fact that the existence of the sup-lattice morphism which completes a diagram like the one in (\ref{concinterdiag}) is guaranteed by the fact that powersets are free (and therefore projective) sup-lattices.
\end{proof}

\section{Concluding remarks}
\label{concl}

The results of Sections \ref{inter} and \ref{interabs} show that sup-lattices and quantale modules provide a good framework for an abstract approach to the comparison of propositional deductive systems. The notations used throughout the paper come from the theory of quantales and their modules, and therefore in some cases they may look less suggestive for the working logician. For this reason, it is worthwhile to remark once again that the sup-lattice $\wp(\SS)_\g$ is nothing else than the lattice of theories $\Th(\vdash_\g)$ of the consequence relation $\vdash_\g$.

On the side of Universal Algebra, $\Fml$ is the term algebra over $\omega$ generators in the signature $\lang$, and it is known that there exists a lattice isomorphism between the lattice of fully invariant $\lang$-congruences on $\Fml$ and the one of equational theories of type $\lang$ (see, for instance, \cite[Chapter II, Section 14]{bursan}). In terms of our notations, each nucleus $\gamma$ on $\wp(\EQ)$ corresponds to a fully invariant congruence $\equiv_\g$ on the $\lang$-algebra $\Fml$. Hence $\wp(\EQ)_\g$ corresponds to the interval $[\equiv_\g,\top]$ of the lattice of fully invariant congruences on $\Fml$. These simple observations should help the reader to better understand the meaning of the results presented.

A remark is needed also for what concerns finitary properties of consequence relations and their interpretations. As the reader may have noticed, we completely disregard finitarity issues throughout the paper. Actually the only reason for that is that the exhaustive discussion presented by Galatos and Tsinakis in Section 6 of \cite{galtsi} readily applies to all the results presented here. Basically, in the correspondence between action-invariant consequence relations and nuclei on quantale modules established in Proposition \ref{oprelnuc}, finitary consequence relations are mapped to algebraic (in the lattice-theoretic sense) nuclei, and vice versa. Hence an interpretation between two systems is finitary (i.e. preserves finitarity) if and only if its corresponding quantale module homomorphism preserves compactness, namely, maps compact elements of the domain to compact elements of the codomain.

The results of Section \ref{interabs}, along with Corollary \ref{transchar2}, allow us to define the category $\cat{ADS}$ of abstract deductive systems as the category whose objects are pairs $\la \QQ, \MM\ra$, where $\QQ$ is a quantale and $\MM$ is a left $\QQ$-module, and whose morphisms are pairs $\la h, f\ra: \la \QQ, \MM \ra \lto \la \RR, \NN\ra$ where $h: \QQ \lto \RR$ is a quantale translation and $f: \MM \lto \NN_h$ is a $\QQ$-module homomorphism. Now let us denote by $p\cat{DS}$ the category whose objects are propositional deductive systems $\la S, \vdash\ra$ and morphisms are pairs composed by a language translation $\tau$ and (the extension to the powerset of the domain of the system of) a $\ov\tau$-action-invariant interpretation. Then it follows immediately from Theorem \ref{faiththm} that there exists a full embedding of $p\cat{DS}$ into $\cat{DS}$. This observation indicates a direction for further investigations and future works.

\end{document}